\documentclass[12pt, reqno, a4paper]{amsart}

\usepackage{amssymb,dsfont,enumerate}
\usepackage{color}
\definecolor{darkgreen}{rgb}{0,0.45,0} 
\usepackage[pagebackref,colorlinks,citecolor=blue,linkcolor=blue, breaklinks]{hyperref}
\usepackage[matrix, arrow, curve]{xy}
\usepackage{stmaryrd}
\usepackage{mathabx}
\usepackage{xcolor}

\usepackage[english]{babel}

\binoppenalty=\maxdimen
\relpenalty=\maxdimen

\theoremstyle{plain}
\newtheorem{theorem}{Theorem}[section]
\newtheorem{lemma}[theorem]{Lemma}
\newtheorem{proposition}[theorem]{Proposition}
\newtheorem{corollary}[theorem]{Corollary}

\theoremstyle{remark}
\newtheorem{remark}[theorem]{Remark}

\theoremstyle{definition}
\newtheorem{example}[theorem]{Example}
\newtheorem{examples}[theorem]{Examples}

\newtheorem{definition}[theorem]{Definition}

\numberwithin{equation}{section}

\DeclareMathOperator{\id}{id}

\DeclareMathOperator{\Ker}{Ker}
\DeclareMathOperator{\im}{im}

\DeclareMathOperator{\End}{End}

\DeclareMathOperator{\supp}{supp}
\DeclareMathOperator{\cosupp}{cosupp}

\DeclareMathOperator{\Hom}{Hom}

\def\Vect{{\bf Vect}}

\def\Act{{\bf Act}}
\def\CoAct{{\bf CoAct}}

\def\ot{\otimes}
\def\ol{\overline}

\begin{document}

\title[$V$-universal Hopf algebras (co)acting on $\Omega$-algebras]{$V$-universal Hopf algebras (co)acting on $\Omega$-algebras}

\author{A.\,L. Agore}
\address{Vrije Universiteit Brussel, Pleinlaan 2, B-1050 Brussels, Belgium}
\address{Simion Stoilow Institute of Mathematics of the Romanian Academy, P.O. Box 1-764, 014700 Bucharest, Romania}
\email{ana.agore@vub.be}

\author{A.\,S. Gordienko}
\address{Department of Higher Algebra,
Faculty of Mechanics and  Mathematics,
M.\,V.~Lomonosov Moscow State University,
Leninskiye Gory, d.1,  Main Building, GSP-1, 119991 Moskva, Russia }
\address{Moscow Center for Fundamental and Applied Mathematics}
\address{Moscow State Technical University of Civil Aviation, Kronshtadtsky boulevard, d.\,20, 125993 Moskva, Russia}

\email{alexey.gordienko@math.msu.ru}

\author{J. Vercruysse}
\address{D\'epartement de Math\'ematiques, Facult\'e des sciences, Universit\'e Libre de Bruxelles, Boulevard du Triomphe, B-1050 Bruxelles, Belgium}
\email{jvercruy@ulb.ac.be}


\begin{abstract} 
We develop a theory which unifies the universal (co)acting bi/Hopf algebras as studied by 
Sweedler, Manin and Tambara with the recently introduced \cite{AGV1} bi/Hopf-algebras that are universal among all support equivalent (co)acting bi/Hopf algebras. Our approach uses vector spaces endowed with a family of linear maps between tensor powers of $A$, called $\Omega$-algebras. This allows us to treat algebras, coalgebras, braided vector spaces and many other structures in a unified way. We study $V$-universal measuring coalgebras and $V$-universal comeasuring algebras between $\Omega$-algebras $A$ and $B$, relative to a fixed subspace $V$ of $\Vect(A,B)$. By considering the case $A=B$, we derive the notion of a $V$-universal (co)acting bialgebra (and Hopf algebra) for a given algebra $A$. In particular, this leads to a refinement of the existence conditions for the Manin--Tambara universal coacting bi/Hopf algebras. We establish an isomorphism between the $V$-universal acting bi/Hopf algebra
and the finite dual of the $V$-universal coacting bi/Hopf algebra under certain conditions on $V$ in terms of the finite topology on $\End_F(A)$.
\end{abstract}


\thanks{The first named author is a fellow of FWO (Fonds voor Wetenschappelijk Onderzoek -- Flanders) and was partially supported by Romanian Ministery of Research and Innovation, CNCS - UEFISCDI, project numbers PN-III-P1-1.1-TE-2016-0124 and PN-III-P4-ID-PCE-2020-0458. The second author is partially supported by a grant of the scientific council of Moscow State Technical University of Civil Aviation and by a grant of the Moscow Center for Fundamental and Applied Mathematics, MSU (Russia). The third author thanks the FNRS for support via the MIS project ``Antipode'' (Grant F.4502.18)}

\maketitle

\tableofcontents

\section{Introduction}

In many areas of mathematics and physics, algebras with an additional structure find their applications, e.g. algebras with an action of a group by automorphisms, algebras with an action of a Lie algebra by derivations
or group graded algebras. For such algebras it is often useful to consider the more general notion of a (co)module algebra over a Hopf algebra which enables to develop theories for different types of additional structures simultaneously. In addition, (co)module algebras can be interpreted as algebras of regular functions on noncommutative spaces with fixed quantum symmetries. As it is often the case in Hopf algebra theory, dual notions play an equally important role and therefore also (co)module coalgebras are studied \cite{Radford}, for example when considering generalized Hopf modules \cite{CMZ}. 

When one studies graded algebras, for many applications, e.g. for the structure theory and graded polynomial identities, it does not matter which specific group is grading the algebra, only the decomposition of the algebra into the direct sum of graded components is important. This idea led naturally to the notion of equivalence of gradings. (See e.g.~\cite{ElduqueKochetov}.) As the authors showed in~\cite{AGV1},
this notion can be generalized further to the case of (co)module structures on algebras.
Equivalence of (co)module structures can be used, for instance, in the study of polynomial $H$-identities.
Furthermore, as the classification of (co)module algebras seems to be a wild problem, a reasonable approach is 
to try to obtain such classification for a given algebra up to equivalence of (co)module structures. Following this avenue of investigation could drastically simplify the process of classification.

 It turns out that among all Hopf algebras (co)acting on a given algebra in an equivalent way, there exists a universal one~\cite[Theorem~3.6, Theorem~4.5]{AGV1}. These universal Hopf algebras are analogues of the universal group of a grading~\cite[Definition~1.17]{ElduqueKochetov}. 
 
 Several other universal (co)acting Hopf algebras/bialgebras are present in the literature. To start with, in the 1960s, M.\,E.~Sweedler introduced the universal measuring coalgebra which leads to the notion of the universal measuring bialgebra~\cite[Chapter VII]{Sweedler}. Unlike the universal Hopf algebra/bialgebra of a specific action defined in~\cite{AGV1}, Sweedler's universal measuring bialgebra is universal among all actions, not necessarily equivalent to a given one. The universal measuring coalgebra was recently shown to exist in the broader context of braided monoidal categories (see \cite{HLV}). Moreover, Sweedler's construction leads naturally to the notion of a universal acting Hopf algebra of an algebra (see Example~\ref{exSw}).

Dual notions of the universal measuring bialgebra were introduced in 1988--1990 by Yu.\,I.~Manin \cite{Manin} and D.~Tambara \cite{Tambara} under the name of universal coacting bi/Hopf algebra of an algebra. For the importance of these objects in non-commutative algebraic geometry we refer the reader to \cite{RVdB2}. It is worth to point out that the existence of the universal coacting (bi)algebra and Hopf algebra has been proved only for finite dimensional algebras \cite[Theorem 1.1]{Tambara}, or for graded algebras that are finite dimensional in each of the homogeneous components \cite{Manin} (remark that in both cases, this means that the algebra is a rigid object in the symmetric monoidal category wherein both the algebra and the universal coacting bi/Hopf algebra live).
We provide examples of (infinite dimensional) algebras and coalgebras for which the universal coacting bi/Hopf algebra does not exist at all in Section~\ref{non-exist} below. 

The present paper was prompted mainly by the desire to build a unified theory, which includes
as particular cases universal Hopf algebras of a given (co)action from~\cite{AGV1} as well
as the Sweedler--Manin--Tambara universal (co)acting bi/Hopf algebras. 
This will be achieved by introducing the notion of a $V$-universal (co)acting bi/Hopf algebra of an algebra $A$ over a field $F$ for a given subalgebra $V$ of the algebra $\End_F(A)$ of linear operators on $A$.
In order to treat in our theory the case of (co)actions both on algebras and on coalgebras (and even on braided vector spaces or other structures) we let $A$ be just a vector space with a fixed set $\Omega$ of linear maps between tensor powers of $A$. We call a space $A$ with such a structure an \textit{$\Omega$-algebra}, see Section~\ref{SectionOmegaAlgebras}. 

Another goal of this paper is to provide a unifying treatment for the different conditions which are classically imposed on an algebra $A$ in order for the Manin--Tambara universal coacting bi/Hopf algebra to exist. This leads to one of the main results of this paper: Theorem~\ref{TheoremBHsquareVExistence}, where necessary and sufficient conditions for the existence of a $V$-universal coacting Hopf algebra are given. In particular, this allows us to refine the conditions of existence of the Manin--Tambara universal coacting bi/Hopf algebras (Corollary~\ref{CorollaryManinExistence}).

Another aspect that we consider in this paper is duality.
In~\cite[Remark 1.3]{Tambara}, D.~Tambara showed that the universal measuring coalgebra from a finite dimensional algebra to an algebra is isomorphic to the finite dual of the universal comeasuring algebra between the same algebras.
In Theorem~\ref{TheoremUnivHopf(Co)actDuality} we provide a refinement of this duality, showing that under certain conditions on $V$ in terms of the finite topology on $\End_F(A)$, the $V$-universal acting bi/Hopf algebra
is isomorphic to the finite dual of the $V$-universal coacting bi/Hopf algebra. An important consequence of this result is that it provides a way of describing the  universal Hopf algebra of an action in the finite dimensional case. Indeed, while the construction of the universal Hopf algebra of a coaction is more or less transparent as it is defined by generators and relations~\cite[Theorem 3.6]{AGV1}, the proof of existence of the universal Hopf algebra of an action is non-constructive. Now with Theorem~\ref{TheoremUnivHopf(Co)actDuality} at hand we can proceed as follows: if $A$ is a finite dimensional $H$-module algebra over a finite dimensional Hopf algebra $H$ via $\psi \colon H\otimes A \to A$, then the universal Hopf algebra of $\psi$ is $H_0^{\circ}$ where we denote by $H_0$ the universal Hopf algebra of the $H^*$-coaction $\rho
\colon A \to A \otimes H^*$ corresponding to $\psi$.

The remainder of the paper is organized as follows.

In Section~\ref{SectionLinearMaps} we introduce some preliminary material on (support) equivalence of linear maps as well as certain topological notions needed throughout.
Section~\ref{SectionOmegaAlgebras} sets the stage for the sequel by introducing the notion of a \textit{$\Omega$-algebra}. In order to make our constructions transparent categorically, in Sections \ref{SectionMeasurings} and \ref{SectionComeasurings} we define (co)measurings in the more general setting of $\Omega$-algebras. In Theorem~\ref{TheoremUnivMeasExistence} we construct a $V$-universal measuring coalgebra, denoted by ${}_\square \mathcal{C}(A,B,V)$, for any $\Omega$-algebras $A$ and $B$ and any subspace $V\subseteq \mathbf{Vect}_F$. The existence of a $V$-universal comeasuring algebra $\mathcal{A}^\square(A,B,V)$ is shown in Theorem~\ref{TheoremUnivComeasExistence} under the assumption that $V$ is closed in the finite topology on $\mathbf{Vect}_F$ and pointwise finite dimensional. The main result of Section~\ref{SectionDuality(Co)measurings} provides a coalgebra isomorphism between ${}_\square \mathcal{C}(A,B,V)$ and the finite dual $\mathcal{A}^\square(A,B,V)^{\circ}$, an analog of the classical duality between Sweedler's measuring coalgebra and Manin--Tambara's universal coacting algebra \cite{Tambara}. 

In Sections~\ref{SectionActions} and~\ref{SectionCoactions} we consider (co)actions in the setting of $\Omega$-algebras by introducing the notions of module $\Omega$-algebras and comodule $\Omega$-algebras over a bialgebra. These generalize the classical (co)module (co)algebras. While $V$-universal acting bi/Hopf algebras (Theorem~\ref{TheoremsquareBBialgebra}) exist for every unital subalgebra
 of $\End_F(A)$, this is not true for $V$-universal coacting bi/Hopf algebras. 
In Section~\ref{SectionCoactions} we show that the latter exist
if $V$ is again closed in the finite topology and pointwise finite dimensional. We denote by ${}_\square \mathcal{B}(A,V)$ and $\mathcal{B}{}^\square (A,V)$ the $V$-universal acting bialgebra, respectively the $V$-universal coacting bialgebra.

The next two sections revolve around the notion of duality. More precisely, Section~\ref{SectionDualityActionsCoactions} studies the duality between actions and coactions. The main result is Theorem~\ref{TheoremUnivHopf(Co)actDuality} where we prove that for any $\Omega$-algebra $A$ and any $V\subseteq \End_F(A)$ a unital pointwise finite dimensional subalgebra closed in the finite topology, the $V$-universal acting bi/Hopf algebra is isomorphic to the finite dual of the $V$-universal coacting bi/Hopf algebra. Finally, in Section~\ref{SectionDualityAlgCoalg} we turn to the classical duality between algebras and coalgebras and study the way $V$-universal (co)acting bialgebras behave with respect to it
in a more general situation of $\Omega$- and $\Omega^*$-algebras. More precisely, we show that for a finite dimensional $\Omega$-algebra $A$ and a unital subalgebra $V \subseteq \End_F(A^*)$ we have a bialgebra isomorphism between ${}_\square \mathcal{B}(A^*,V)$ and ${}_\square \mathcal{B}(A,V^{\#})^{\mathrm{op}}$, where $V^{\#} := \lbrace f^* \mid f \in V \rbrace$. Similarly, we have a bialgebra isomorphism between $\mathcal{B}^\square(A^*,V)$ and $\mathcal{B}^\square(A,V^{\#})^{\mathrm{cop}}$. The two isomorphisms can be extended to the level of $V$-universal (co)acting Hopf algebras if we assume some extra conditions on ${}_\square \mathcal{B}(A,V^{\#})$ and $\mathcal{B}^\square(A,V^{\#})$. This is done in Theorems~\ref{dual3} and~\ref{dual4}.

Finally, in Section~\ref{non-exist} we present explicit examples of (co)algebras for which the universal coacting Hopf algebra does not exist. 

We end this introductory part with a few words about notation. 
 Throughout the paper we fix the base field $F$. A vector space will always mean a vector space over $F$, the same for linear maps, tensor products, (co)algebras, bialgebras etc.
All coalgebras considered here are counital and coassociative. $A^{\circ}$ stands for the finite dual of an algebra $A$. The multiplication and the unit maps of an algebra $A$ are denoted by $\mu_{A}$ and $u_{A}$. Similarly, $\Delta_{C}$ and $\varepsilon_{C}$ are used to designate the comultiplication map and the counit map of a coalgebra $C$. When there is no fear of confusion we drop the subscripts. For a bialgebra $B$ we denote by $B^\mathrm{op}$ (resp. $B^\mathrm{cop}$) the opposite (resp. co-opposite) bialgebra. 
We use the following notation for the usual categories: $\mathbf{Alg}_F$ (unital associative algebras over $F$), $\mathbf{Coalg}_F$ (coalgebras over $F$), $\mathbf{Bialg}_F$ (bialgebras over $F$), $\mathbf{Hopf}_F$ (Hopf algebras over $F$) and $\mathbf{SHopf}_F$ (Hopf algebras with bijective antipode over $F$).
  
  If $A$ and $B$ are objects of a category $\mathcal A$, we denote by $\mathcal A(A,B)$ the set
  of all morphisms $A \to B$ in $\mathcal A$. For example, $\mathbf{Vect}_F(V,W)$ is the set of all
  linear maps $V \to W$ for vector spaces $V$ and $W$.
  
  We usually omit the notation for forgetful and embedding functors, e.g. $\mathbf{Alg}_F \to \mathbf{Vect}_F$. In addition, we freely use the vector space structure on $\mathbf{Vect}_F(V,W)$ without special notice. If $S$ is a subset of a vector space over a field $F$, we denote the linear span of $S$ by $\langle S \rangle_F$.

\section{Linear maps, their (co)supports, and finite topology}\label{SectionLinearMaps}

\subsection{Support and cosupport of linear maps}
\label{SubsectionLinearMapsMod}

We start by introducing some notation and terminology that will be of use in our study of (co)measurings in the next sections.\\

Let $P$ be any vector space. We denote by $\Act_P$ the category whose objects are triples $(A,B,\psi)$ where $A$ and $B$ are vector spaces and $\psi\colon P\ot A\to B$ is a linear map. A morphism $(A,B,\psi)\to(A',B',\psi')$ in $\Act_P$ is a pair of linear maps $f\colon A\to A'$, $g\colon B\to B'$ such that $\psi'\circ (\id_P \ot f)=g\circ \psi$. Of course if $P$ has no additional structure, objects in $\Act_P$ are not ``actions'' in the usual sense.

Let $(A,B,\psi)$ be an object in $\Act_P$. Then let us denote by $\ol\psi \colon P\to \Vect_F(A,B)$ the associated morphism under the Hom-tensor adjunction, in other words, $\ol\psi(p)(a)=\psi(p\ot a)$ for all $p\in P$ and $a\in A$.
\begin{definition}
We call the subspace $$\cosupp \psi := \psi(P \otimes (-)) =\ol\psi(P) \subseteq \mathbf{Vect}_F(A,B)$$
 the \textit{cosupport} of $\psi$. 
 \end{definition}

\begin{definition}
Suppose $\psi_i \colon P_i \otimes A \to B$, $i=1,2$, are two linear maps for some vector spaces $P_1$
and $P_2$.
We say that $\psi_1$ is \textit{coarser} than $\psi_2$ and 
$\psi_2$ is \textit{finer} than $\psi_1$ and write $\psi_1 \preccurlyeq \psi_2$
if $\cosupp \psi_1 \subseteq \cosupp \psi_2$. 
We say that $\psi_1$ and $\psi_2$ are \textit{(support) equivalent} if $\psi_1 \preccurlyeq \psi_2$
and $\psi_1 \succcurlyeq \psi_2$, i.e.\ when they have the same cosupport.
\end{definition}

\label{SubsectionLinearMapsComod}

Let $Q$ be any vector space. We denote by $\CoAct_Q$ the category whose objects are triples $(A,B,\rho)$ where $A$ and $B$ are vector spaces and $\rho\colon A\to B\ot Q$ is a linear map. A morphism $(A,B,\rho)\to(A',B',\rho')$ in $\CoAct_Q$ is a pair of linear maps $f\colon A\to A'$, $g\colon B\to B'$ such that $\rho'\circ f=(g\ot \id_Q)\circ \rho$. 
Analogously to comodule structures, for such maps we use the Sweedler notation $\rho(a)=a_{(0)}\otimes a_{(1)}$ for $a\in A$, where the summation symbol is omitted.

Let $(a_\alpha)_\alpha$ be a basis in $A$ and let $(b_\beta)_\beta$ be a basis in $B$.
Define $q_{\beta\alpha} \in Q$ by $\rho(a_\alpha)=\sum_{\beta} b_\beta \otimes q_{\beta\alpha}$.
We call the linear span of  $q_{\beta\alpha}$ the \textit{support} of $\rho$
and denote it by $\supp \rho$.
Applying linear functions dual to $b_\beta$, it is easy to see that $\supp \rho$ does not depend on the choice of the bases in $A$ and $B$. In fact, $\supp \rho$ coincides with the smallest subspace $Q_1$ of $Q$ such that $\rho(A)\subseteq B \otimes Q_1$, i.e.\ such that $\rho$ factors as in the following diagram
\[
\xymatrix{
A \ar[rr]^-\rho \ar[dr] && B\ot Q\\
&B\ot\supp\rho \ar[ur]
}
\]

For any map $\rho\colon A \to B \otimes Q$ we denote by $\rho^{\vee} \colon Q^* \otimes A \to B$ the linear map defined
by $\rho^{\vee}(q^*\otimes a):= q^*(a_{(1)})a_{(0)}$. Obviously, this construction defines a functor $\CoAct_Q\to \Act_{Q^*}$ that sends $(A,B,\rho)$ to $(A,B,\rho^{\vee})$ and any morphism to itself.
\begin{definition}
We call the subspace $$\cosupp \rho := \cosupp \rho^{\vee} = \rho^{\vee}(Q^* \otimes (-)) =\ol{\rho^{\vee}}(Q^*) \subseteq \mathbf{Vect}_F(A,B)$$
the \textit{cosupport} of $\rho$.
\end{definition}

\begin{definition}
Suppose $\rho_i \colon A \to B \otimes Q_i$, $i=1,2$, are linear maps for some vector spaces $Q_i$.
We say that $\rho_1$ is \textit{coarser} than $\rho_2$ and 
$\rho_2$ is \textit{finer} than $\rho_1$ and write $\rho_1 \preccurlyeq \rho_2$
if $\cosupp \rho_1 \subseteq \cosupp \rho_2$. 
We say that $\rho_1$ and $\rho_2$ are \textit{(support) equivalent} if $\rho_1 \preccurlyeq \rho_2$
and $\rho_1 \succcurlyeq \rho_2$.
\end{definition}

Let us observe that the notions of support and cosupport as introduced above are indeed dual notions, as the next lemma makes precise.

\begin{lemma}\label{lemmasup}
Let $\rho\colon A \to B \otimes Q$ be a linear map and $\ol{\rho^{\vee}}\colon Q^*\to\Vect_F(A,B)$ as defined above. Then there is a (unique) linear isomorphism 
$$\theta\colon (\supp \rho)^* \to \cosupp\rho$$
making the following diagram, which expresses the epi-mono factorization of $\ol\rho$, commutative:
\[\xymatrix{
Q^* \ar@{->>}[rr]^\pi \ar[d]_{\ol{\rho^{\vee}}} &&  (\supp \rho)^* \ar[d]^{\theta}_\cong\\
\Vect_F(A,B) && \cosupp \rho \ar@{_(->}[ll]
}\]
where $\pi$ is the dual of the inclusion map $\supp\rho\subseteq Q$
\end{lemma}

\begin{proof}
Denote by $Q'$ any linear complement of $\supp\rho$ in $Q$, then $Q\cong \supp\rho\oplus Q'$ and hence also $Q^*\cong (\supp\rho)^*\oplus {Q'}^*$. Using this direct sum decomposition, we immediately obtain that $\ol{\rho^{\vee}}({Q'}^*)=0$ and therefore $\cosupp\rho = \ol{\rho^{\vee}}(Q^*)=\ol{\rho^{\vee}}((\supp\rho)^*)$. Hence the corestriction of $\ol{\rho^{\vee}}$ to $(\supp\rho)^*$ is well defined and surjective. Explicitly, this map is given by $\theta\colon (\supp\rho)^*\to \cosupp\rho,\ \theta(f)(a)=a_{(0)}f(a_{(1)})$ for all $f\in (\supp\rho)^*$ and $a\in A$.

It remains to check that $\theta$ is injective. To this end, consider as above a basis $(a_\alpha)_\alpha$  in $A$, a basis $(b_\beta)_\beta$ in $B$ and write $\rho(a_\alpha)=\sum_\beta b_\beta\ot q_{\beta\alpha}$, such that the elements $q_{\alpha\beta}$ generate $\supp\rho$. Suppose that $\theta(f)=0$ for some $f\in(\supp\rho)^*$. Then we find for all $\alpha$ that 
$\sum_\beta b_\beta f(q_{\beta\alpha})=0$ Since the elements $b_\beta$ form a base this implies that $f(q_{\beta\alpha})=0$ for all choices of $\alpha$ and $\beta$, hence $f=0$.
\end{proof}

As a consequence of the previous lemma, support equivalence can be expressed equivalently by means of the support and of the cosupport. More precisely, we have the following.

\begin{proposition}\label{PropositionPrecLinearMapComodCriterion}
Consider linear maps $\rho_i \colon A \to B \otimes Q_i$, $i=1,2$. 
Then $\rho_1 \preccurlyeq \rho_2$ if and only if there exists a linear map $\tau\colon
\supp \rho_2 \to \supp \rho_1$ making the diagram below commutative:
\begin{equation}\label{EqLinearMapComodABTauQ} \xymatrix{ A \ar[r] 
\ar[rd]
 & B \otimes \supp \rho_2 \ar@{-->}[d]^{\id_B \otimes \tau} \\
& B \otimes \supp \rho_1} \end{equation}
The map $\tau$, if it exists, is unique and surjective.
Consequently, the linear maps $\rho_1$ and $\rho_2$ are equivalent if and only if $\tau$ is a linear isomorphism.
\end{proposition}
\begin{proof}
If $\rho_1 \preccurlyeq \rho_2$, then $\cosupp \rho_1 \subseteq \cosupp \rho_2$. Combining (the dual of) this inclusion with the isomorphisms $\theta_i \colon (\supp \rho_i)^*\to \cosupp \rho_i$ from Lemma \ref{lemmasup}, we obtain an injective linear map $\tilde\tau \colon (\supp\rho_1)^*\to (\supp\rho_2)^*$ that sends an element $f\in(\supp\rho_1)^*$ to the (unique) element $g\in(\supp\rho_2)^*$ satisfying $(\id_B\ot f)\rho_1(a)=(\id_B\ot g)\rho_2(a)$ for all $a\in A$.
\[
\xymatrix{
(\supp\rho_1)^*\ar[rr]^{\tilde\tau} \ar[d]_{\theta_1} && (\supp\rho_2)^* \\
\cosupp\rho_1 \ar@{^(->}[rr] && \cosupp \rho_2 \ar[u]_{\theta_2^{-1}}
}
\]
If we choose as above a basis $(a_\alpha)_\alpha$  in $A$ and a basis $(b_\beta)_\beta$ in $B$, and define elements $q_{\beta\alpha} \in Q_1$ and $p_{\beta\alpha} \in Q_2$ by $\rho_1(a_\alpha)=\sum_{\beta} b_\beta \otimes q_{\beta\alpha}$ and  $\rho_2(a_\alpha)=\sum_{\beta} b_\beta \otimes p_{\beta\alpha}$, then \begin{equation}\label{EqTauF}\tilde\tau(f)(p_{\beta\alpha})=f(q_{\beta\alpha}).\end{equation}
Consider the restriction $\tau$ of the dual map $\tilde\tau^* \colon (\supp\rho_2)^{**}\to (\supp\rho_1)^{**}$ to $\supp \rho_2$ treated as a subspace of $(\supp\rho_2)^{**}$.
By~\eqref{EqTauF}  we have $\tau(p_{\alpha\beta})=q_{\alpha\beta}$,
which means that we have a correctly defined linear map $\tau\colon \supp\rho_2\to \supp\rho_1$. The map $\tau$ is obviously surjective and makes the diagram~\eqref{EqLinearMapComodABTauQ} commutative.

Conversely, suppose that there exists a linear map $\tau\colon
\supp \rho_2 \to \supp \rho_1$ making the diagram~\eqref{EqLinearMapComodABTauQ} commutative. Then $\theta_2\circ \tau^*\circ \theta_1^{-1}$ defines an inclusion of $\cosupp\rho_1$ into $\cosupp\rho_2$.
Hence $\rho_1 \succcurlyeq \rho_2$.

Finally, if $\rho_1$ and $\rho_2$ are equivalent, then
$\rho_2 \succcurlyeq \rho_1$ implies the existence of a unique map $\tau\colon
\supp \rho_2 \to \supp \rho_1$  making the diagram~\eqref{EqLinearMapComodABTauQ} commutative,
while $\rho_1 \succcurlyeq \rho_2$ implies the existence of the map 
$\theta \colon \supp \rho_1 \to \supp \rho_2$ making the diagram
$$ \xymatrix{ A \ar[r] 
\ar[rd]
 & B \otimes \supp \rho_2  \\
& B \otimes \supp \rho_1 \ar@{-->}[u]_{\id_B \otimes \theta}} $$  commutative. Since both $\theta\tau$
and $\id_{\supp \rho_2}$ correspond to $\rho_2 \succcurlyeq \rho_2$
and $\tau\theta$
and $\id_{\supp \rho_1}$ correspond to $\rho_1 \succcurlyeq \rho_1$,
their uniqueness implies $\theta\tau = \id_{\supp \rho_2}$
and $\tau\theta = \id_{\supp \rho_1}$. Hence $\tau$ is indeed a linear isomorphism.
\end{proof}

\subsection{Closed pointwise finite dimensional subspaces in $\mathbf{Vect}_F(A,B)$}\label{SubsectionClLocFinSubspaces}

For any subspace $V\subseteq \Vect_F(A,B)$, the restriction of the evaluation map defines a linear map $\psi\colon V\ot A\to B$ such that $\cosupp\psi=V$. 
In this section we establish necessary and sufficient conditions (see Theorem~\ref{TheoremVectDualization} below) for a subspace $V \subseteq \mathbf{Vect}_F(A,B)$ to be ``dualizable'', i.e. to be equal to $\cosupp \rho$ for some linear map $\rho \colon A \to B \otimes Q$, or, equivalently (in view of Lemma \ref{lemmasup}), for $V$ to be isomorphic to the dual space of $\supp \rho$.\\

Recall that the {\it finite topology} on $\Vect_F(A,B)$ is the topology generated by the base of open sets given by all subsets of the form
$$W_{a_1,\ldots, a_n; b_1,\ldots,b_n}=\lbrace f \in \mathbf{Vect}_F(A,B) \mid
f(a_1)=b_1,\ldots, f(a_n)=b_n \rbrace$$ where $n\in\mathbb N$, $a_1,\ldots, a_n \in A$,
$b_1,\ldots,b_n \in B$. As this turns $\Vect_F(A,B)$ into a topological group, these open sets are also closed. Remark that the finite topology on $\Vect_F(A,B)$ is exactly the restriction of the compact-open topology on ${\bf Set}(A,B)$, where both $A$ and $B$ are regarded with the discrete topology.

\begin{lemma}\label{LemmaDualizationClosed}
For any vector spaces $A,B,Q$ and a linear map $\rho \colon A \to B \otimes Q$
the subspace $\rho^{\vee}(Q^* \otimes (-))\subseteq \mathbf{Vect}_F(A,B)$ is closed in the finite topology.
\end{lemma}
\begin{proof}
Let $(a_\alpha)_\alpha$ be a basis in $A$ and let $(b_\beta)_\beta$ be a basis in $B$.
Again, define $q_{\beta\alpha}\in Q$ by $\rho(a_\alpha)=\sum_{\beta} b_\beta \otimes q_{\beta\alpha}$.
(For a given $\alpha$ only a finite number of $q_{\beta\alpha}$ is nonzero.)

Fix $f \in \mathbf{Vect}_F(A,B)$ and define $f_{\beta\alpha} \in F$ by
$f(a_\alpha)=\sum_{\beta} f_{\beta\alpha} b_\beta$. (Again, for a given $\alpha$ only a finite number of $f_{\beta\alpha}$ is nonzero.)
Then $f=\rho^{\vee}(q^* \otimes (-))$ for some $q^* \in Q^*$
if and only if 
\begin{equation}\label{EqSystemFqalphabeta} 
q^*(q_{\beta\alpha})= f_{\beta\alpha}\text{ for all }\alpha\text{ and }\beta.
\end{equation}

In order to finish the proof, we have to show that for every such $f$ one can choose $n\in\mathbb N$ and $\alpha_1, \ldots, \alpha_n$, $\beta_1, \ldots, \beta_n$, i.e. a finite number of indices,
such that there is no $q^* \in Q^*$ with $q^*(q_{\beta_i\alpha_i})= f_{\beta_i\alpha_i}$
for all $1\leqslant i \leqslant n$. This would mean that the open neighbourhood $$\lbrace
g \in \mathbf{Vect}_F(A,B) \mid g(a_{\alpha_1})=f(a_{\alpha_1}),
\ldots, g(a_{\alpha_n})=f(a_{\alpha_n}) \rbrace$$ of $f$ has empty intersection
with $\rho^{\vee}(Q^* \otimes (-))$.

Suppose that for all finite subsets $\alpha_1, \ldots, \alpha_n$, $\beta_1, \ldots, \beta_n$
of indices the system of linear equations \begin{equation*}
q^*(q_{\beta_i\alpha_i})= f_{\beta_i\alpha_i},\ 1\leqslant i \leqslant n,\end{equation*} has
a solution $q^* \in Q^*$. Then for every $\lambda_i \in F$ such that $\sum_{i=1}^n \lambda_i q_{\beta_i\alpha_i} = 0$ we have $\sum_{i=1}^n \lambda_i f_{\beta_i\alpha_i} = 0$.
In other words, we may exclude equations with left hand sides that are linearly dependent from the 
left hand sides of the other equations.
Denote by $(q_\gamma)_\gamma$ a maximal linearly independent subsystem of $(q_{\beta\alpha})_{\alpha,\beta}$, which exists by the Zorn lemma. Denote $f_\gamma:=f_{\beta\alpha}$
for $q_\gamma = q_{\beta\alpha}$.
Then the system 
\begin{equation*} q^*(q_\gamma)= f_\gamma\text{ for all }\gamma\end{equation*}
is equivalent to \eqref{EqSystemFqalphabeta}.
Since $q_\gamma$ are linearly independent, we may include $(q_\gamma)_\gamma$
into a basis in $Q$ and the system \eqref{EqSystemFqalphabeta} has a solution.
In this case $f \in \rho^{\vee}(Q^* \otimes (-))$.

Therefore, for a given $f \in \mathbf{Vect}_F(A,B)$ we have only two possibilities: either $f \in \rho^{\vee}(Q^* \otimes (-))$ or $f$ does not belong to $\rho^{\vee}(Q^* \otimes (-))$ together with some open neighbourhood. Hence $\rho^{\vee}(Q^* \otimes (-))$ is indeed closed in the finite topology.
\end{proof}

\begin{definition}
We say that a subspace $V\subseteq \mathbf{Vect}_F(A,B)$ is \textit{pointwise finite dimensional} if for every $a\in A$
the subspace $Va := \lbrace va \mid v\in V \rbrace \subseteq B$ is finite dimensional.
\end{definition}

\begin{lemma}\label{LemmaDualizationPointwiseFinDim}
For any vector spaces $A,B,Q$ and a linear map $\rho \colon A \to B \otimes Q$
the subspace $\rho^{\vee}(Q^* \otimes (-))\subseteq \mathbf{Vect}_F(A,B)$ is pointwise finite dimensional.
\end{lemma}
\begin{proof} For every $a\in A$ there exists $n\in\mathbb N$, $a_i \in A$,
and $q_i \in Q$ such that
$$\rho(a)=a_1\otimes q_1 + \dots + a_n \otimes q_n.$$ Hence $\rho^{\vee}(q^*\otimes a)=
\sum_{k=1}^n q^*(q_j)a_j \in \langle a_1, \ldots, a_n \rangle_F$ for every $q^* \in Q^*$
and $\dim_F(Va)\leqslant n$.
\end{proof}

The following result is well-known, but we add a short proof for the reader's convenience.

\begin{lemma}\label{LemmaVVstarstarOnFinite} Let $V$ be a vector space. Then $V$ is dense in $V^{**}$ with respect to the finite topology. In other words, for any finite dimensional subspace $T \subseteq V^*$ and any $v^{**} \in V^{**}$ there exists $v \in V$ such that $v\bigr|_{T}$ (viewed as an element of $V^{**}$) coincides with $v^{**}\bigr|_{T}$.
\end{lemma}
\begin{proof}
Choosing a basis in $V$, we can represent any $v^{*} \in V^{*}$ as a (possibly, infinite)
row of the values of $v^{*}$ on the basis elements.
Choose, in addition, a basis $v^*_1, \ldots, v^*_n$ in $T$. Performing Gauss--Jordan elimination
on the rows corresponding to  $v^*_1, \ldots, v^*_n$, we can pass to another basis
$\tilde v^*_1, \ldots, \tilde v^*_n$ in $V^{*}$ such that the corresponding rows contain an identity $n\times n$ submatrix. Taking the basis elements $v_1, \ldots, v_n$ corresponding to the columns of this submatrix, we define $v:= \sum_{i=1}^n v^{**}(\tilde v^*_i)v_i$.
\end{proof}

Now we are ready to prove the main theorem of this section characterizing dualizable subspaces in 
$V \subseteq \mathbf{Vect}_F(A,B)$

\begin{theorem}\label{TheoremVectDualization}
Let $V \subseteq \mathbf{Vect}_F(A,B)$ be a subspace for 
 vector spaces $A$ and $B$. Then $V=\rho^{\vee}(Q^* \otimes (-))$
for some vector space $Q$ and a linear map $\rho \colon A \to B \otimes Q$
if and only if $V$ is pointwise finite dimensional and closed in the finite topology.
\end{theorem}
\begin{proof}
The ``only if'' part was proved in Lemmas~\ref{LemmaDualizationClosed} and~\ref{LemmaDualizationPointwiseFinDim}. 

Conversely, suppose $V$ is pointwise finite dimensional and closed in the finite topology. We will prove the statement by putting $Q=V^*$.
Again choose bases $(a_\alpha)_\alpha$ in $A$ and $(b_\beta)_\beta$ in $B$.
Define $\rho  \colon A \to B \otimes V^*$
as follows. For a given $\alpha$ the pointwise finite dimensionality of $V$ implies
that $Va_\alpha \subseteq \langle b_{\beta_1}, \ldots, b_{\beta_n} \rangle_F$
for some $n\in\mathbb N$ and $\beta_1, \ldots, \beta_n$.
Then there exist $v_i^* \in V^*$ such that
$f(a_\alpha)=\sum_{i=1}^n v_i^*(f)b_{\beta_i}$ for all $f\in V$.
Let $\rho(a_\alpha) := \sum_{i=1}^n b_{\beta_i} \otimes v_i^*$.

Obviously, $V \subseteq \rho^{\vee}(V^{**} \otimes (-))$.
Let us prove the converse inclusion. Let $v^{**} \in V^{**}$.
We claim that $\rho^{\vee}(v^{**} \otimes (-)) \in V$.

Let $\rho(a_\alpha) = \sum_{\beta} b_{\beta} \otimes v^*_{\beta\alpha}$.
(For a given $\alpha$ only a finite number of $v^*_{\beta\alpha}$ is nonzero.)
Suppose $\rho^{\vee}(v^{**} \otimes (-)) \notin V$. Since
$V$ is closed in the finite topology,
there exist $n\in\mathbb N$ and $\alpha_1, \ldots, \alpha_n$
such that for any $f \in V$ 
we have $$\sum_{\beta} v^*_{\beta\alpha_i}(f) b_{\beta}=f(a_{\alpha_i}) \ne \rho^{\vee}(v^{**} \otimes a_{\alpha_i})= \sum_{\beta} v^{**}(v^*_{\beta\alpha_i}) b_{\beta}
\text{ for some }1\leqslant i \leqslant n.$$
In other words,
the system 
\begin{equation}\label{EqSystemFVstaralphabeta} v^*_{\beta\alpha_i}(f)=v^{**}(v^*_{\beta\alpha_i})\text{ for all }\beta\text{ and all }1\leqslant i \leqslant n\end{equation}
has no solution in $f \in V$. 
Since for a fixed $i$ only a finite number of $v^*_{\beta\alpha_i}$ is nonzero,
the system~\eqref{EqSystemFVstaralphabeta} is finite.
Lemma~\ref{LemmaVVstarstarOnFinite} implies that we get a contradiction and $\rho^{\vee}(v^{**} \otimes (-)) \in V$.
Hence $V = \rho^{\vee}(V^{**} \otimes (-))$.
\end{proof}

\begin{remark}
If both $A$ and $B$ are finite dimensional, then the finite topology on $\mathbf{Vect}_F(A,B)$
is discrete and all subspaces of $\mathbf{Vect}_F(A,B)$ are closed and pointwise finite dimensional.
\end{remark}

The lemmas below will be used in Sections~\ref{SectionDuality(Co)measurings} and~\ref{SectionCoactions}.
\begin{lemma}\label{LemmaClosurePwFD}
Let $V \subseteq \mathbf{Vect}_F(A,B)$ be a subspace for some vector spaces $A$ and $B$ over
a field $F$. If $V$ is pointwise finite dimensional,
then its closure $\overline V$ in the finite topology is pointwise finite dimensional too.
\end{lemma}
\begin{proof}
Let $f\in \overline V$. Then for any $n\in\mathbb N$ and $a_1, \ldots, a_n \in A$ there exists $g\in V$
such that $f(a_i)=g(a_i)$, $1\leqslant i \leqslant n$.
Hence for every $a\in A$ we have $\overline V a = Va$ and $\overline V$ is pointwise finite dimensional too.
\end{proof}

\begin{lemma}\label{LemmaLinearMap(Co)modClosure}
Let $A,B,P$ be vector spaces and $\rho \colon A \to B \otimes P^*$ a linear map, $\rho(a)=a_{(0)}\otimes a_{(1)}$
for $a\in A$. Define the linear map $\psi \colon P \otimes A \to B$ by $\psi(p\otimes a):=a_{(1)}(p)a_{(0)}$
for all $a\in A$, $p\in P$.
Then $\cosupp \rho = \overline{\cosupp \psi}$ (the closure is taken in the finite topology).
\end{lemma}
\begin{proof}
Obviously, $\cosupp \psi \subseteq \cosupp \rho$. By Lemma~\ref{LemmaDualizationClosed},
$\cosupp \rho$ is closed. Hence it is sufficient to prove that for every 
$f \in \cosupp \rho$, $n\in\mathbb N$, $a_1, \ldots, a_n \in A$ there exists $g \in \cosupp \psi$
such that $f(a_i)=g(a_i)$ for all $1\leqslant i \leqslant n$. 

Again choose bases $(a_\alpha)_\alpha$ in $A$ and $(b_\beta)_\beta$ in $B$
and define  $q_{\beta\alpha} \in P^*$ by $\rho(a_\alpha)=\sum_{\beta} b_\beta \otimes q_{\beta\alpha}$.
By Lemma~\ref{LemmaVVstarstarOnFinite},
 for every $n\in\mathbb N$, indices $\alpha_1, \ldots, \alpha_n$, and $p^{**}\in P^{**}$,
 there exists $p\in P$ such that
  $q_{\beta\alpha_i}(p)=p^{**}(q_{\beta\alpha_i})$
for all $1\leqslant i \leqslant n$ and $\beta$  
  (here we again use that for a fixed $i$ only finite number of $q_{\beta\alpha}$
 is nonzero). Hence the values of the linear maps $A \to B$
 corresponding to $p$ and $p^{**}$ coincide on $a_{\alpha_1}, \ldots, a_{\alpha_n}$.
 Since any element of $A$ can be written as a finite linear combination of some $a_\alpha$,
 the lemma follows.
\end{proof}

\section{Measurings and comeasurings between $\Omega$-algebras}

\subsection{$\Omega$-algebras}\label{SectionOmegaAlgebras}\

Let $\Omega$ be a {\em span} from $\mathbb Z_+$ to $\mathbb Z_+$ in $\bf Sets$. In other words, $\Omega$ is a set endowed with two maps $s,t\colon \Omega\to \mathbb Z_+$ (which we will call {\em source} and {\em target}). 

\begin{definition}
We call an \textit{$\Omega$-algebra} a vector space $A$ endowed with linear maps $\omega_A \colon A^{\otimes s(\omega)} \to
A^{\otimes t(\omega)}$ for every $\omega \in \Omega$. (We will usually drop the subscript $A$ and denote
 the map by $\omega$ too.) Here use the convention that $A^{\otimes 0} := F$.
\end{definition}

A \textit{morphism} of $\Omega$-algebras $f\colon A\to B$ is a linear map satisfying 
$$f^{\ot t(\omega)}\circ \omega_A=\omega_B\circ f^{\ot s(\omega)}$$
for all $\omega\in\Omega$: 
\[
\xymatrix{
A^{\ot s(\omega)} \ar[rr]^-{f^{\ot s(\omega)}} \ar[d]_{\omega_A} && B^{\ot s(\omega)} \ar[d]^{\omega_B}\\
A^{\ot t(\omega)} \ar[rr]^-{f^{\ot t(\omega)}}  && B^{\ot t(\omega)} 
}
\]
The category of $\Omega$-algebras is denoted by $\Omega\text{-}\mathbf{Alg}_F$. 

The above definitions make sense in any monoidal category (instead of $\Vect_F$).

\begin{examples} 
\begin{enumerate}
\item
$\varnothing$-algebras are just vector spaces.
\item
\label{ExampleAlgebra} $\lbrace \mu \rbrace$-algebras where $s(\mu)=2$, $t(\mu)=1$
are just not necessarily associative not necessarily unital algebras $A$
with a multiplication $\mu \colon A  \otimes A \to A$.
\item \label{ExampleUnitalAlgebra} Unital algebras $A$
with a multiplication $\mu \colon A  \otimes A \to A$
and an identity element $1_A$
are examples of $\lbrace \mu, u \rbrace$-algebras where $s(\mu)=2$, $t(\mu)=1$, $s(u)=0$, $t(u)=1$
and $u \colon F \to A$ is defined by $u(\alpha)=\alpha 1_A$ for all $\alpha \in F$.
\item \label{ExampleCoalgebra} Coalgebras $C$
with a comultiplication $\Delta \colon C \to C  \otimes C$
and a counit $\varepsilon \colon C \to F$
are examples of $\lbrace \Delta, \varepsilon \rbrace$-algebras where $s(\Delta)=1$, $t(\Delta)=2$, $s(\varepsilon)=1$, $t(\varepsilon)=0$.
\item
Braided vector spaces $W$
with a braiding $\tau \colon W \otimes W \to W \otimes W$
are examples of $\lbrace \tau \rbrace$-algebras where $s(\tau)=2$, $t(\tau)=2$.
\item
Bialgebras, Hopf algebras, Frobenius algebras, quantum torsors, ternary algebras, etc.\ are examples of $\Omega$-algebras too for appropriate $\Omega$. 
\end{enumerate}
\end{examples}

Obviously, tensor products of $\Omega$-algebras are again $\Omega$-algebras.
For a span $\Omega$, we denote by $\Omega^*$ the dual span, which is the span obtained by interchanging the roles of source and target.
Namely, for a given set $\Omega$ together with maps $s,t\colon\Omega \to \mathbb Z_+$
define the set $\Omega^* := \left\lbrace \omega^* \mid \omega\in\Omega \right\rbrace$
and the maps $s^* \colon \Omega^* \to \mathbb Z_+$ and $t^* \colon \Omega^* \to \mathbb Z_+$
where $s^*(\omega^*):=t(\omega)$ and $t^*(\omega^*):=s(\omega)$
for all $\omega \in \Omega$.
 Then the classical duality between algebras and coalgebras is naturally extended to the duality between $\Omega$- and $\Omega^*$-algebras:
    
\begin{proposition}
\begin{enumerate}[(i)]
\item If $A$ is an $\Omega$-algebra such that $\im s\subseteq\{0,1\}$, then $A^*$ is an $\Omega^*$-algebra
where the operations on $A^*$ are duals of the corresponding operations on $A$. More precisely, this construction defines a contravariant functor $\Omega\text{-}\mathbf{Alg}_F\to \Omega^*\text{-}\mathbf{Alg}_F$. 
\item 
If $A$ is an $\Omega$-algebra that is finite dimensional as a vector space, then $A^*$ is an $\Omega^*$-algebra. More precisely, this construction defines a contravariant equivalence of categories $\Omega\text{-}\mathbf{Alg}_F^\mathrm{fd}\to \Omega^*\text{-}\mathbf{Alg}_F^\mathrm{fd}$, where the index $fd$ means that we  consider the categories of finite-dimensional $\Omega$-algebras. 
\end{enumerate}
\end{proposition}    

\begin{remark} One can define the notion of a variety of $\Omega$-algebras in the following natural way. Let a \textit{derived operation} be a formal linear combination of compositions of some of $\omega \in \Omega$ together, possibly, with some flips of components (of course, only linear combinations of compositions of the same arity and coarity are allowed). Let $\tilde\Omega$ be the set of all derived operations induced by $\Omega$
and let $P\subseteq \tilde\Omega$ be a subset.
Then the \textit{variety defined by the set $P$ of polynomial identities} is the class of all $\Omega$-algebras where the symbols from $P$ correspond to zero maps.
All varieties of ordinary algebras (in particular, Lie, associative commutative) as well as the classes
of counital coassociative coalgebras, bialgebras and Hopf algebras are varieties in this sense.
The results of this article however do not depend on the variety to which a given $\Omega$-algebra belongs and hence are applicable to all cases.
\end{remark}

\subsection{Measurings}\label{SectionMeasurings}

Let us start by making the following observation
\begin{proposition}
Let $P$ be a coalgebra. Then the category $\Act_P$ is a monoidal category such that both forgetful functors $U_1,U_2\colon\Act_P\to \Vect_F$ defined on objects by $U_1(A,B,\psi)=A$ and $U_2(A,B,\psi)=B$ are strict monoidal.
\end{proposition}
\begin{proof}
Given two objects $(A,B,\psi)$ and $(A',B',\psi')$ in $\Act_P$, we define the tensor product as $(A\ot A',B\ot B',\psi'')$, where we consider:  
$$\psi''(p\ot a\ot a')=\psi(p_{(1)}\ot a)\ot \psi'(p_{(2)}\ot a').$$
The monoidal unit is given by $(F,F,\varepsilon)$ where $\varepsilon$ is the counit of coalgebra $P$, regarded as a map $P\cong P\ot F\to F$. 
One can easily verify that this indeed endows $\Act_P$ with the structure of a monoidal category.
\end{proof}

For any $(A,B,\psi)$ in $\Act_P$ we can in particular consider the $n$-fold tensor product of this object with itself for any $n\in \mathbb N$. We will denote this by $(A,B,\psi)^{\ot n}=(A^{\ot n},B^{\ot n},\psi_n)$, where we have (by the above)
$$\psi_n(p\otimes a_1 \otimes \dots \otimes a_n):= \psi(p_{(1)} \otimes a_1) \otimes \dots \otimes \psi(p_{(n)}\otimes a_n)$$
for  all $a_i \in A$, $p\in P$, $n>0$ and $\psi_0(p)=\varepsilon_P(p)$.

We now come to one of the main definitions of this paper.

\begin{definition}
A linear map $\psi\colon P\ot A\to B$, where $A$ and $B$ are $\Omega$-algebras and $P$ is a coalgebra, is called a {\em measuring} from $A$ to $B$ if and only if $(A,B,\psi)$ is an $\Omega$-algebra in the monoidal category $\Act_P$. Spelled out, $\psi$ is a measuring if and only if for all $\omega\in\Omega$ 
we have $$\psi_{t(\omega)}(p \otimes \omega_A(a_1 \otimes \dots \otimes a_{s(\omega)}))
=\omega_B(\psi_{s(\omega)}(p \otimes a_1 \otimes \dots \otimes a_{s(\omega)}))$$
for all $a_i \in A$, $p\in P$.
\end{definition}

\begin{examples} 
\begin{enumerate}
\item If $\Omega = \lbrace \mu, u \rbrace$
and $A$, $B$ are unital algebras, we get the usual definition~\cite[Chapter VII]{Sweedler} of a measuring from $A$ to $B$:
$$\psi(p \otimes ab)= \psi(p_{(1)} \otimes a)\psi(p_{(2)}\otimes b) 
\text{ and }\psi(p\otimes 1_A)= \varepsilon(p)1_B$$
for $a,b \in A$ and $p\in P$.
\item  If $A$ and $B$ are coalgebras, then a measuring from $A$ to $B$ is just a coalgebra homomorphism $P \otimes A \to B$.
\end{enumerate}
\end{examples}

\begin{remark}
If $(A,B,\psi)$ is in $\mathbf{Act}_P$, we know that there is an associated morphism $P\to \Vect(A,B)$. Under this correspondence, one easily sees that if $\psi$ is a measuring then grouplike elements of $P$ correspond to morphisms of $\Omega$-algebras. In the same way, primitive elements in $P$ could be interpreted as a kind of generalized derivations. 
\end{remark}

The following result shows that measurings can transport interesting information between $\Omega$-algebras.

\begin{proposition}
Let $A$ and $B$ be $\{\mu\}$-algebras (i.e. vector spaces with a binary operation). If $B$ is associative and there exists a measuring $\psi\colon P\ot A\to B$ such that $\psi$ is injective, then $A$ is associative as well.
\end{proposition}
\begin{proof}
It suffices to see that $\psi(p\ot a(bc))=\psi(p\ot (ab)c)$ for all $p\in P$ and $a,b,c\in A$. 
\end{proof}

Fix $\Omega$-algebras $A$ and $B$ and a subspace $V\subseteq \mathbf{Vect}_F(A,B)$. Consider the category $\mathbf{Meas}(A,B,V)$ where the objects are pairs $(P,\psi)$ consisting of a coalgebra $P$ and a measuring $\psi \colon P \otimes A \to B$  such that $\cosupp \psi \subseteq V$. A morphism from $\psi_1 \colon P_1 \otimes A \to B$
to $\psi_2 \colon P_2 \otimes A \to B$ in $\mathbf{Meas}(A,B,V)$ is a coalgebra homomorphism $\varphi \colon P_1 \to P_2$
making the following diagram commutative:
$$\xymatrix{ P_1 \otimes A \ar[r]^(0.6){\psi_1} \ar[d]_{\varphi \otimes \id_A} & B \\
P_2 \otimes A  \ar[ru]_{\psi_2}  } $$

\begin{theorem}\label{TheoremUnivMeasExistence} In $\mathbf{Meas}(A,B,V)$ there exists a terminal object, that we denote by $({}_\square \mathcal{C}(A,B,V),\psi_{A,B,V})$ and that we call the \textit{$V$-universal measuring coalgebra} from $A$ to $B$.
\end{theorem}
\begin{proof}
Denote by $K$ the right adjoint functor for the forgetful functor $\mathbf{Coalg}_F \to \mathbf{Vect}_F$.
In other words, for a given vector space $W$ the coalgebra $K(W)$ is the \textit{cofree coalgebra} of $W$.
Denote by ${}_\square \mathcal{C}(A,B,V)$ the sum of all subcoalgebras $C$ of $K(V)$ such that
the composition of the embedding $C \subseteq K(V)$, the counit $K(V)\to V$ of the adjunction
and the embedding $V \subseteq \mathbf{Vect}_F(A,B)$
defines a measuring. Then, by the universal property of $K(V)$, the map $\psi_{A,B,V} \colon {}_\square \mathcal{C}(A,B,V) \otimes A \to B$
corresponding to the composition of the embedding $ {}_\square\mathcal{C}(A,B,V) \subseteq K(V)$, the counit $K(V)\to V$ of the adjunction and the embedding $V \subseteq \mathbf{Vect}_F(A,B)$ is the terminal object of 
$\mathbf{Meas}(A,B,V)$ and ${}_\square \mathcal{C}(A,B,V)$ is the $V$-universal measuring coalgebra from $A$ to $B$. 
\end{proof}

\begin{remark} It is easy to see that the cosupport of the universal measuring $\psi_{A,B,V}$
is a subspace of $V$ and coincides with $V$ if there exists a measuring from $A$ to $B$ with the cosupport equal to $V$. 
\end{remark}

\begin{example}
If $A$ and $B$ are unital associative algebras, the coalgebra ${}_\square\mathcal{C}(A,B,\mathbf{Vect}_F(A,B))$ is exactly the Sweedler universal measuring coalgebra~\cite[Chapter VII]{Sweedler}.
\end{example}

\subsection{Comeasurings}\label{SectionComeasurings}

Dually to what we did in the previous subsection, one can see that for any unital associative algebra $Q$, the category  $\CoAct_Q$ is a monoidal category such that both forgetful functors to $\Vect$ are strict monoidal. If an object object $(A,B,\rho)$ in this monoidal category has an $\Omega$-algebra structure, then we say that $Q$ comeasures from $A$ to $B$. Let us spell out this definition in detail.

Let $A$ and $B$ be $\Omega$-algebras and let $Q$ be a unital associative algebra.
Suppose $\rho \colon A \to B \otimes Q$ is a linear map and denote $\rho(a)=a_{(0)}\ot a_{(1)}$, for any $a\in A$.
Then for $n\in\mathbb Z_+$ we have induced maps $\rho_n \colon A^{\otimes n} \to B^{\otimes n} \otimes Q$
defined as follows: $\rho_0(\alpha)=\alpha 1_Q$ for $\alpha \in F$,
and $$\rho_n(a_1 \otimes \dots \otimes a_n):= (a_1)_{(0)} \otimes \dots \otimes (a_{n})_{(0)}
\otimes (a_1)_{(1)}  \dots  (a_{n})_{(1)}$$
for all $a_i \in A$.

\begin{definition}
 A linear map $\rho \colon A \to B \otimes Q$ is called a \textit{comeasuring} from $A$ to $B$ if it preserves the operations from $\Omega$, i.e. for every  $\omega \in \Omega$,
$a_i \in A$
we have $$\rho_{t(\omega)} \left( \omega_A(a_1 \otimes \dots \otimes a_{s(\omega)}) \right)
= (\omega_B \otimes \id_Q)\rho_{s(\omega)}(a_1 \otimes \dots \otimes a_{s(\omega)}).$$
\end{definition}

 \begin{example} In the case $\Omega = \lbrace \mu, u \rbrace$ and $A$, $B$ are unital algebras we get the following definition of a comeasuring from $A$ to $B$:
 
\hspace{1.7cm}$\rho(1_A)= 1_B \otimes 1_Q$ and $\rho(ab)= a_{(0)}b_{(0)}\otimes a_{(1)}b_{(1)}$  for all $a,b \in A$.\\
In other words, $\rho$ is an algebra homomorphism.
\end{example}

\begin{remark} The restriction of $\rho^{\vee}$ (see the definition in Section~\ref{SectionLinearMaps})
to $Q^\circ \otimes A$, where the coalgebra $Q^\circ$ is the finite dual of the algebra $Q$, is a measuring.
If $Q$ is finite dimensional, then $\rho^{\vee}$ itself is a measuring and $\cosupp \rho = \cosupp \rho^{\vee}$.
\end{remark}

Fix $\Omega$-algebras $A$ and $B$ and a subspace $V\subseteq \mathbf{Vect}_F(A,B)$. Consider the category $\mathbf{Comeas}(A,B,V)$ where the objects are all comeasurings $\rho \colon A \to B \otimes Q$ for all unital associative algebras $Q$ with $\cosupp \rho \subseteq V$ and the morphisms from $\rho_1 \colon A \to B \otimes Q_1$
to $\rho_2 \colon A \to B \otimes Q_2$  are unital algebra homomorphisms $\varphi \colon Q_1 \to Q_2$
making the diagram below commutative:

$$\xymatrix{ A \ar[r]^(0.4){\rho_1} 
\ar[rd]_{\rho_2}
 & B \otimes Q_1 \ar[d]^{\id_B \otimes \varphi} \\
& B \otimes Q_2} $$

Let us call the algebra $\mathcal{A}^\square(A,B,V)$ corresponding to the initial object in $\mathbf{Comeas}(A,B,V)$ (if it exists) the \textit{$V$-universal comeasuring algebra} from $A$ to $B$.

\begin{theorem}\label{TheoremUnivComeasExistence}  If $V$ is pointwise finite dimensional and closed in the finite topology, there exists an initial object in $\mathbf{Comeas}(A,B,V)$.
\end{theorem}
\begin{proof} By Theorem~\ref{TheoremVectDualization}, there exists a vector space
$W$ and a linear map $\rho_0 \colon A \to B \otimes W$ such that $\rho^{\vee}_0(W^* \otimes (-)) = V$.

Let $(a_\alpha)_\alpha$ be a basis in $A$ and let $(b_\beta)_\beta$ be a basis in $B$.
Define $p_{\beta\alpha} \in W$ by $$\rho_0(a_\alpha)=\sum_{\beta} b_\beta \otimes p_{\beta\alpha}.$$
Let $W_0 := \supp \rho_0=\langle p_{\beta\alpha}  \mid \alpha, \beta \rangle_F$.

 For every $\omega \in \Omega$ define $a^{\beta_1,\ldots,\beta_{t(\omega)}}_{\omega;\alpha_1,\ldots,\alpha_{s(\omega)}}, b^{\beta_1,\ldots,\beta_{t(\omega)}}_{\omega;\alpha_1,\ldots,\alpha_{s(\omega)}}\in F$
by
$$\omega_A(a_{\alpha_1} \otimes \dots \otimes a_{\alpha_{s(\omega)}})
= \sum_{\beta_1,\ldots,\beta_{t(\omega)}} a^{\beta_1,\ldots,\beta_{t(\omega)}}_{\omega;\alpha_1,\ldots,\alpha_{s(\omega)}}a_{\beta_1} \otimes \dots \otimes a_{\beta_{t(\omega)}}$$
and 
$$\omega_B(b_{\alpha_1} \otimes \dots \otimes b_{\alpha_{s(\omega)}})
= \sum_{\beta_1,\ldots,\beta_{t(\omega)}} b^{\omega;\beta_1,\ldots,\beta_{t(\omega)}}_{\alpha_1,\ldots,\alpha_{s(\omega)}}b_{\beta_1} \otimes \dots \otimes b_{\beta_{t(\omega)}}.$$

In other words, $a^{\beta_1,\ldots,\beta_{t(\omega)}}_{\omega;\alpha_1,\ldots,\alpha_{s(\omega)}}$
and $b^{\beta_1,\ldots,\beta_{t(\omega)}}_{\omega;\alpha_1,\ldots,\alpha_{s(\omega)}}$
are the structure constants of the $\Omega$-algebras $A$ and $B$, respectively.

Denote by $T$ the left adjoint functor for the forgetful functor $\mathbf{Alg}_F \to \mathbf{Vect}_F$.
In other words, for a given vector space $M$ the algebra $T(M)$ is the \textit{free (or tensor) algebra} of $M$. Let $i_M \colon M \to T(M)$ be the unit of this adjunction. Denote by $I$ the ideal of $T(W_0)$ generated by the elements
\begin{equation*}\begin{split} \sum_{\gamma_1,\ldots,\gamma_{t(\omega)}} a^{\gamma_1,\ldots,\gamma_{t(\omega)}}_{\omega;\alpha_1,\ldots,\alpha_{s(\omega)}}
i_{W_0}(p_{\beta_1 \gamma_1})i_{W_0}(p_{\beta_2 \gamma_2})
\dots i_{W_0}(p_{\beta_{t(\omega)} \gamma_{t(\omega)}}) \\
- \sum_{\mu_1,\ldots,\mu_{s(\omega)}}
b^{\beta_1,\ldots,\beta_{t(\omega)}}_{\omega;\mu_1,\ldots,\mu_{s(\omega)}}
i_{W_0}(p_{\mu_1 \alpha_1})i_{W_0}(p_{\mu_2 \alpha_2})\dots i_{W_0}(p_{\mu_{s(\omega)} \alpha_{s(\omega)}})\end{split}\end{equation*}
 for all $\omega \in \Omega$ and $\alpha_1,\ldots,\alpha_{s(\omega)}$, $\beta_1,\ldots,\beta_{t(\omega)}$ and $\omega \in \Omega$.

Let $ \mathcal{A}^\square(A,B,V):=T(W_0)/I$ and let
$\rho_{A,B,V} \colon A \to B
\otimes \mathcal{A}^\square(A,B,V)$ be the composition of the corestriction of $\rho_0$
to $B \otimes W_0$ and the maps $i_{W_0} \colon W_0 \to T(W_0)$ and $T(W_0)\twoheadrightarrow T(W_0)/I$ tensored by $\id_B$.
Then the choice of $I$ implies that $\rho_{A,B,V}$ is a comeasuring.
Suppose $\rho \colon A \to B \otimes Q$ is another comeasuring with $\cosupp \rho \subseteq V$
where $Q$ is a unital associative algebra. Then Proposition~\ref{PropositionPrecLinearMapComodCriterion} implies
that there exists a linear map $\tau \colon W_0 \to Q$ making the following
diagram commutative:
\begin{equation*} \xymatrix{ A \ar[r] 
\ar[rd]_\rho
 & B \otimes W_0 \ar@{-->}[d]^{\id_B \otimes \tau} \\
& B \otimes Q} \end{equation*}

As $Q$ is a unital associative algebra, there exists an algebra homomorphism
$\tau_1 \colon T(W_0) \to Q$ making the diagram 

\begin{equation*} \xymatrix{ A \ar[r] 
\ar[rd]_\rho
 & B \otimes T(W_0) \ar@{-->}[d]^{\id_B \otimes \tau_1} \\
& B \otimes Q} \end{equation*} commutative.

Since $\rho$ is a comeasuring we have $I \subseteq \Ker(\tau_1)$ and
 there exists an algebra homomorphism
$\bar\tau_1 \colon \mathcal{A}^\square(A,B,V) \to Q$
making the diagram
\begin{equation*} \xymatrix{ A \ar[rr]^(0.3){\rho_{A,B,V}}
\ar[rrd]_\rho
 & & B \otimes \mathcal{A}^\square(A,B,V) \ar@{-->}[d]^{\id_B \otimes \bar\tau_1} \\
& & B \otimes Q} \end{equation*}
 commutative.

Now recall that $\mathcal{A}^\square(A,B,V)$ is generated by the images of $p_{\beta\alpha}$. Therefore, 
the homomorphism $\bar \tau_1$ is unique and $\rho_{A,B,V}$ is an initial object in $\mathbf{Comeas}(A,B,V)$.
\end{proof}

\begin{remark} It is easy to see that the cosupport of $\rho_{A,B,V} \colon A \to B
\otimes \mathcal{A}^\square(A,B,V)$
is a subspace of $V$ (the cosupport could become a proper subspace of $V$ after the factorization by $I$) and coincides with $V$ if there exists a comeasuring from $A$ to $B$ with the cosupport equal to $V$. 
\end{remark}

\begin{examples}
\begin{enumerate}
\item If $A$ and $B$ are unital associative algebras with $B$ finite dimensional then $\mathbf{Vect}_F(A,B)$ is obviously pointwise finite dimensional and closed in the finite topology and $\mathcal{A}^\square(A,B,\mathbf{Vect}_F(A,B))$ coincides with the Tambara universal coacting algebra \cite[Section 1]{Tambara}. \\
\item Let $A$ and $B$ be unital associative algebras such that $\mathcal{A}^\square(A,B, \mathbf{Vect}_F(A,B))$ exists. If, moreover, $A$ is a bialgebra and $B$ is commutative then $\mathcal{A}^\square(A,B, \mathbf{Vect}_F(A,B))$ is a bialgebra. Indeed, to start with, note that $\omega_{A,B,\mathbf{Vect}_F(A,B)} \colon A \to B \otimes \mathcal{A}^\square(A,B, \mathbf{Vect}_F(A,B)) \otimes \mathcal{A}^\square(A,B, \mathbf{Vect}_F(A,B))$ and $\lambda \colon A \to B \otimes F$ defined below are comeasurings:
\begin{eqnarray*}
&& \omega_{A,B,\mathbf{Vect}_F(A,B)} = (\mu_{B} \otimes \id \otimes \id) (\id \otimes \tau \otimes \id) (\rho_{A,B,\mathbf{Vect}_F(A,B)} \otimes \rho_{A,B,\mathbf{Vect}_F(A,B)}) \Delta_{A}\\
&& \lambda = {\rm can} \, u_{B} \, \varepsilon_{A}
\end{eqnarray*}
where ${\rm can} \colon B \to B \otimes F$ is the canonical isomorphism. The comultiplication $\Delta$ and counit $\varepsilon$ on $\mathcal{A}^\square(A,B, \mathbf{Vect}_F(A,B))$ are the unique algebra homomorphisms such that:
\begin{eqnarray*}
&& (\id \otimes \Delta) \rho_{A,B,\mathbf{Vect}_F(A,B)} = \omega_{A,B,\mathbf{Vect}_F(A,B)}\\
&& (\id \otimes \varepsilon)  \rho_{A,B,\mathbf{Vect}_F(A,B)} = \lambda
\end{eqnarray*}
It is straightforward to see that $\Delta$ and $\varepsilon$ define a bialgebra structure on $\mathcal{A}^\square(A,B, \mathbf{Vect}_F(A,B))$. A similar result was obtained in \cite[Section 3]{Mastnak} in the context of measurings.\\
\item Let $A$, $B$ be associative algebras, $Q$ a unital associative algebra, and $\rho \colon A \to B \otimes Q$ a comeasuring. Let $(a_\alpha)_\alpha$ be a basis in $A$, $(b_\beta)_\beta$ a basis in $B$ and define $q_{\beta\alpha} \in Q$ by $\rho(a_\alpha)=\sum_{\beta} b_\beta \otimes q_{\beta\alpha}$. Consider $I$ to be the ideal of the tensor algebra $T(\supp \rho)$ generated by 
$$
\sum_{r,s} c^\gamma_{rs}\, i_{\supp \rho}(q_{r\alpha})\, i_{\supp \rho}(q_{s\beta}) - \sum_{u} k^{u}_{\alpha\beta}\, i_{\supp \rho} (q_{\gamma u})
$$
for all possible choices of indices $\alpha,\beta,\gamma$, where $i_{\supp \rho} \colon \supp \rho \to T(\supp \rho)$ is the canonical inclusion and $k^{u}_{\alpha\beta} \in F$
and $c^\gamma_{rs} \in F$ denote the structure constants of $A$ and $B$, respectively. Then $\mathcal{A}^\square(A,B, \cosupp \rho) = T(\supp \rho)/I$ and $\rho_{A,B,V} = (\id_{B} \otimes\, \pi\,  i_{\supp \rho}) \rho$, where $\pi \colon T(\supp \rho) \to T(\supp \rho)/I$ is the canonical projection. 

Furthermore, it can be easily proved that $\bigl(\mathcal{A}^\square(A,B, \cosupp \rho), \rho_{A,B,V}\bigl)$ is the initial object in the category whose objects are all comeasurings equivalent to $\rho$ while the morphisms between two such objects $\rho_{1} \colon A \to B \otimes Q_{1}$ and $\rho_{2} \colon A \to B \otimes Q_{2}$ are all algebra homomorphisms $\tau \colon Q_{1} \to Q_{2}$ such that $(\id_{B} \otimes \tau) \rho_{1} = \rho_{2}$.
\end{enumerate}
\end{examples}

\subsection{Duality between measurings and comeasurings}\label{SectionDuality(Co)measurings}

In this section we establish the duality between ${}_\square \mathcal{C}(A,B,V)$ and 
$\mathcal{A}^\square(A,B,V)$. 

We first prove the following lemma:

\begin{lemma}\label{LemmaMeasComeasBijection}
Let $A$ and $B$ be $\Omega$-algebras and let $V \subseteq \mathbf{Vect}_F(A,B)$ be a pointwise finite dimensional subspace closed in the finite topology.
Denote by $\mathrm{Meas}(P,A,B,V)$
the set of measurings $\psi \colon P \otimes A \to B$  with $\cosupp \psi \subseteq V$
and by $\mathrm{Comeas}(P^*,A,B,V)$
the set of comeasurings $\rho \colon A \to B \otimes P^*$  with $\cosupp \rho \subseteq V$.
Then we have a bijection $\mathrm{Meas}(P,A,B,V) 
\cong \mathrm{Comeas}(P^*,A,B,V)$
natural in the coalgebra $P$ if we regard $\mathrm{Meas}(-,A,B,V)$ and $\mathrm{Comeas}((-)^*,A,B,V)$
as functors $\mathbf{Coalg}_F \to \mathbf{Sets}$.
\end{lemma}
\begin{proof} Every $\rho \in \mathrm{Comeas}(P^*,A,B,V)$ induces a measuring $\rho^\vee \colon P \otimes A \to B$ where $\rho^\vee(p\otimes a):=a_{(1)}(p)a_{(0)}$ for all $a\in A$ and $p\in P$.
Then $\cosupp \rho^\vee \subseteq \cosupp \rho \subseteq V$ and $\rho^\vee \in \mathrm{Meas}(P,A,B,V)$.

Conversely, if $\psi \in \mathrm{Meas}(P,A,B,V)$, then $\cosupp \psi$ is pointwise finite dimensional.
Hence for every $a\in A$ there exist $n\in\mathbb N$, $b_1, \ldots, b_n \in B$, $p_1^*, \ldots, p_n^* \in P^*$
such that for every $p\in P$ we have $\psi(p\otimes a)=\sum_{i=1}^n p_i^*(p)b_i$.
Note that for fixed linearly independent $b_1, \ldots, b_n$ linear functions
$p_1^*, \ldots, p_n^* \in P^*$ are defined uniquely. Hence there exists a unique comeasuring $\rho \colon A \to B \otimes P^*$ such that $\rho^\vee = \psi$. The values of $\rho$ are defined by
$\rho(a)=\sum_{i=1}^n b_i \otimes p_i^*$.
Since $V$ is closed, by Lemma~\ref{LemmaLinearMap(Co)modClosure} we have $\cosupp \rho = \overline{\cosupp \psi} \subseteq V$
and $\rho \in \mathrm{Comeas}(P^*,A,B,V)$.
\end{proof}

Recall that for an algebra $A$ the finite dual of $A$ is denoted by $A^\circ$.

\begin{theorem}\label{TheoremUniv(Co)measDuality} Let $A$ and $B$ be $\Omega$-algebras and let $V \subseteq \mathbf{Vect}_F(A,B)$ be a pointwise finite dimensional subspace closed in the finite topology.  Denote the restriction of $\widetilde{\rho_{A,B,V}} \colon 
\mathcal{A}^\square(A,B,V)^* \otimes A \to B$ on $\mathcal{A}^\square(A,B,V)^\circ \otimes A$
by $\widetilde{\rho_{A,B,V}}$ too. If $\theta$ is the unique coalgebra homomorphism making
the diagram
$$\xymatrix{ {}_\square \mathcal{C}(A,B,V) \otimes A \ar[dr]^(0.6){\psi_{A,B,V}}& \\
                   &   B \\
\mathcal{A}^\square(A,B,V)^\circ \otimes A \ar[ru]_(0.6){\widetilde{\rho_{A,B,V}}} \ar@{-->}[uu]^{\theta \otimes \id_A} } $$  commutative,
then $\theta$ is a coalgebra isomorphism.
\end{theorem}
\begin{proof}
Inspired by~\cite[Remark~1.3]{Tambara},
we can add the bijection from Lemma~\ref{LemmaMeasComeasBijection} to the following bijections natural in the coalgebra $P$:
\begin{equation}\label{EqTambaraBijections}\begin{split}\mathbf{Coalg}_F(P, {}_\square \mathcal{C}(A,B,V))\cong \mathrm{Meas}(P,A,B,V) \\ \cong \mathrm{Comeas}(P^*,A,B,V)\cong
\mathbf{Alg}_F(\mathcal{A}^\square(A,B,V), P^*)
\cong \mathbf{Coalg}_F(P, \mathcal{A}^\square(A,B,V)^\circ).\end{split}
\end{equation}

Now if we substitute for $P$ the coalgebra $\mathcal{A}^\square(A,B,V)^\circ$,
the coalgebra homomorphism $\theta \colon \mathcal{A}^\square(A,B,V)^\circ \to {}_\square \mathcal{C}(A,B,V)$
will correspond to $\id_{\mathcal{A}^\square(A,B,V)^\circ}$.
If we substitute for $P$ the coalgebra ${}_\square \mathcal{C}(A,B,V)$,
by the naturality,
the homomorphism $\id_{{}_\square \mathcal{C}(A,B,V)}$ 
will correspond to $\theta^{-1}$. Hence $\theta$ is a coalgebra isomorphism.
\end{proof}

\begin{remark} One can formulate this duality in the form of an adjuction too.
Let us call a measuring $\psi \colon P \otimes A \to B$ \textit{rational}
if $\psi = \rho^{\vee}_0\bigr|_{P \otimes A}$ for some linear map $\rho_0 \colon A \to B \otimes P^*$.
(Here as usual we treat $P \otimes A$ as a subspace of $P^{**} \otimes A$.)
Denote $\hat\psi := \rho_0$
which is defined uniquely for every rational $\psi$ and is in fact a comeasuring. In addition, denote by $\mathbf{Meas}^{\mathrm{rat}}(A,B,V)$ the full subcategory 
of $\mathbf{Meas}(A,B,V)$ whose objects are all rational measurings $\psi \colon P \otimes A \to B$.
Then for every subspace $V \subseteq \mathbf{Vect}_F(A,B)$ closed in finite topology, rational measuring $\psi \colon P \otimes A \to B$ and
comeasuring $\rho \colon A \to B \otimes Q$, such that $\cosupp \psi, \cosupp\rho \subseteq V$,
there is a natural bijection $$\mathbf{Meas}^{\mathrm{rat}}(A,B,V)(\psi, \rho^{\vee}\bigr|_{Q^\circ \otimes A}) \cong \mathbf{Comeas}(A,B,V)(\rho, \hat\psi)$$
(a restriction of the bijection $\mathbf{Coalg}_F(P,Q^\circ)\cong \mathbf{Alg}_F(Q,P^*)$) which means that the corresponding contravariant functors are adjoint.
Here we get an alternative proof of Theorem~\ref{TheoremUniv(Co)measDuality}.
(It is sufficient to use the contravariant version of the fact that right adjoint covariant functors preserve terminal objects.)
\end{remark}

\section{Universal (co)actions on $\Omega$-algebras}

\subsection{Actions}\label{SectionActions}

If $B$ is a bialgebra, then the category of left $B$-modules is a monoidal category. Unpacking the definition of an $\Omega$-algebra therein, we obtain the following.

\begin{definition}
Let $B$ be a bialgebra and let $A$ be an $\Omega$-algebra. We say that 
a linear map $\psi \colon B \otimes A \to A$
is a \textit{$B$-action} on $A$ or, which is the same, $\psi$ defines on $A$ a structure of a \textit{(left) $B$-module $\Omega$-algebra} if the following two conditions hold: \begin{enumerate}
\item $\psi$ defines on $A$ a structure of a (unital) left $B$-module;
\item $\psi$ is a measuring from $A$ to $A$.
\end{enumerate}
\end{definition}

If $\Omega=\lbrace\mu, u\rbrace$ and $A$ is a unital algebra, then we get the usual definition of a unital module algebra.
If $\Omega=\lbrace\mu\rbrace$, then the algebra is not necessarily unital.
Finally, if $\Omega=\lbrace\Delta, \varepsilon\rbrace$ and $A$ is a coalgebra, then we get the definition of a module
coalgebra.

For actions we can use the notion of cosupport, equivalence, and the preorder introduced in Section~\ref{SectionLinearMaps}.

Note that for an action $\psi \colon B \otimes A \to A$ its cosupport $\cosupp\psi$ is
just the image of the corresponding algebra homomorphism $B \to \End_F(A)$. 
In particular, $\cosupp\psi$ is a unital subalgebra in $\End_F(A)$.

\begin{theorem}\label{TheoremsquareBBialgebra}
Let $A$ be an $\Omega$-algebra and let $V\subseteq \End_F(A)$ be a unital subalgebra.
Then the $V$-universal measuring coalgebra ${}_\square \mathcal{B}(A,V) := {}_\square \mathcal{C}(A,A,V)$ admits a structure of a bialgebra such that for any bialgebra $B$ and any action
$\psi \colon B \otimes A \to A$ such that $\cosupp \psi \subseteq V$ the unique coalgebra homomorphism $\varphi$
making the diagram below commutative is in fact a bialgebra homomorphism:
\begin{equation}\label{EqTheoremsquareBBialgebra}\xymatrix{ B \otimes A \ar[r]^(0.6){\psi} \ar@{-->}[d]_{\varphi \otimes \id_A} & A \\
{}_\square \mathcal{B}(A,V) \otimes A  \ar[ru]_{\psi_{A,V}}  } \end{equation}
(Here $\psi_{A,V} := \psi_{A,A,V}$.)
\end{theorem}
We call ${}_\square \mathcal{B}(A,V)$ the \textit{$V$-universal acting bialgebra}.
\begin{proof}[Proof of Theorem~\ref{TheoremsquareBBialgebra}]
Note that $\psi_{A,V}(\id_{{}_\square \mathcal{B}(A,V)} \otimes \psi_{A,V})
\colon {}_\square \mathcal{B}(A,V)\otimes {}_\square \mathcal{B}(A,V) \otimes A \to A$
is a measuring too and $\cosupp(\psi_{A,V}(\id_{{}_\square \mathcal{B}(A,V)} \otimes \psi_{A,V}))
\subseteq V$ since $V$ is a subalgebra.
Hence there exists a unique coalgebra homomorphism $\mu \colon 
{}_\square \mathcal{B}(A,V)\otimes {}_\square \mathcal{B}(A,V) \to {}_\square \mathcal{B}(A,V)$
making the following diagram commutative:
$$\xymatrix{ {}_\square \mathcal{B}(A,V)\otimes {}_\square \mathcal{B}(A,V) \otimes A \ar[rrr]^(0.6){\id_{{}_\square \mathcal{B}(A,V)} \otimes \psi_{A,V}} \ar@{-->}[d]_{\mu \otimes \id_A} && & {}_\square \mathcal{B}(A,V) \otimes A  \ar[d]^{\psi_{A,V}} \\
{}_\square \mathcal{B}(A,V) \otimes A  \ar[rrr]_{\psi_{A,V}} &&  & A  } $$
Now we define the multiplication in ${}_\square \mathcal{B}(A,V)$ using $\mu$.

In order to define the identity element we consider the trivial map $$F \otimes A
\mathrel{\widetilde\to}  A,$$ which is a measuring.
There exists a unique coalgebra homomorphism $u \colon F \to {}_\square \mathcal{B}(A,V)$
making the diagram below commutative:
$$\xymatrix{ F \otimes A \ar[r]^(0.6){\sim} \ar@{-->}[d]_{u \otimes \id_A} & A \\
{}_\square \mathcal{B}(A,V) \otimes A  \ar[ru]_{\psi_{A,V}}  } $$
We define $1_{{}_\square \mathcal{B}(A,V)} := u(1_F)$.

Now we have to prove that the multiplication $\mu$ is associative and that $1_{{}_\square \mathcal{B}(A,V)}$
is indeed the identity element of ${}_\square \mathcal{B}(A,V)$.

Consider the diagram

{\tiny
$$\xymatrix{  {}_\square \mathcal{B}(A,V)\otimes {}_\square \mathcal{B}(A,V) \otimes {}_\square\mathcal{B}(A,V) \otimes A 
\ar[rr]^{ \id_{{}_\square \mathcal{B}(A,V)} \otimes \mu\otimes \id_A} \ar[dd]^(0.7){\id_{{}_\square \mathcal{B}(A,V)} \otimes\id_{{}_\square \mathcal{B}(A,V)} \otimes\psi_{A,V}} \ar[rd]^(0.6){\qquad\mu\otimes \id_{{}_\square \mathcal{B}(A,V)} \otimes \id_A} &&  {}_\square \mathcal{B}(A,V) \otimes {}_\square\mathcal{B}(A,V) \otimes A \ar'[d][dd]^{\id_{{}_\square \mathcal{B}(A,V)} \otimes\psi_{A,V}} \ar[rd]^{\mu\otimes \id_A} & \\
& {}_\square \mathcal{B}(A,V) \otimes {}_\square\mathcal{B}(A,V) \otimes A \ar[rr]^(0.4){\mu\otimes \id_A} \ar[dd]^(0.7){\id_{{}_\square \mathcal{B}(A,V)} \otimes\psi_{A,V}}
&& {}_\square\mathcal{B}(A,V) \otimes A \ar[dd]^{\psi_{A,V}} \\
{}_\square \mathcal{B}(A,V) \otimes {}_\square\mathcal{B}(A,V) \otimes A \ar'[r][rr]^(0.4){\id_{{}_\square \mathcal{B}(A,V)} \otimes\psi_{A,V}} \ar[rd]^{\mu\otimes \id_A} && {}_\square\mathcal{B}(A,V) \otimes A \ar[rd]^{\psi_{A,V}} & \\
& {}_\square\mathcal{B}(A,V) \otimes A \ar[rr]_{\psi_{A,V}} && A } $$ }

The left face is commutative since both compositions are equal to 
$\mu \otimes \psi_{A,V}$. The commutativity of the other faces except the upper one follows from
the definition of $\mu$.
Therefore the compositions of the upper face being composed with $\psi_{A,V}$
are equal too.
Now the universal property of ${}_\square \mathcal{B}(A,V)$ implies the commutativity of the diagram
 $$\xymatrix{ {}_\square \mathcal{B}(A,V)\otimes {}_\square \mathcal{B}(A,V) \otimes {}_\square\mathcal{B}(A,V) \ar[rrr]^(0.6){\id_{{}_\square \mathcal{B}(A,V)} \otimes \mu} \ar[d]_{\mu \otimes \id_{{}_\square \mathcal{B}(A,V)}} && & {}_\square \mathcal{B}(A,V) \otimes {}_\square\mathcal{B}(A,V)  \ar[d]^{\mu} \\
{}_\square \mathcal{B}(A,V) \otimes {}_\square \mathcal{B}(A,V)  \ar[rrr]_{\mu} &&  & {}_\square\mathcal{B}(A,V)  } $$
and the associativity follows.

Consider the diagram

$$\xymatrix{ F \otimes {}_\square \mathcal{B}(A,V)\otimes A \ar[rr]^{u \otimes \id_{{}_\square \mathcal{B}(A,V)}\otimes\id_A} \ar[rd]^{\sim} \ar[dd]_(0.7){\id_F  \otimes\psi_{A,V}}
  & & {}_\square \mathcal{B}(A,V)\otimes {}_\square \mathcal{B}(A,V)\otimes A  \ar[dd]^{\id_{{}_\square \mathcal{B}(A,V)} \otimes \psi_{A,V}}  \ar[ld]^{\mu\otimes\id_A}\\
 & {}_\square \mathcal{B}(A,V)\otimes A \ar[dd]^(0.3){\psi_{A,V}} & \\
F \otimes A \ar'[r][rr]^(0.3){u\otimes\id_A} \ar[rd]^{\sim}
  & & {}_\square \mathcal{B}(A,V)\otimes  A  \ar[ld]^{\psi_{A,V}}\\
 &  A  & 
 } $$
 
 The commutativity of the rear face and the left face is obvious. The commutativity of the lower face
 follows from the definition of $u$ and the commutativity of the right face follows from the definition
 of $\mu$. Therefore, the coalgebra homomorphisms forming the upper face become equal after their
 composition with $\psi_{A,V}$. 
 Now universal property of ${}_\square \mathcal{B}(A,V)$ implies the commutativity of the diagram
 $$\xymatrix{ F \otimes {}_\square \mathcal{B}(A,V) \ar[rr]^{u \otimes \id_{{}_\square \mathcal{B}(A,V)}} \ar[rd]^{\sim} 
  & & {}_\square \mathcal{B}(A,V)\otimes {}_\square \mathcal{B}(A,V) \ar[ld]^{\mu}\\
 & {}_\square \mathcal{B}(A,V)  & }$$
 Hence $1_{{}_\square \mathcal{B}(A,V)}$ is the left identity element of ${}_\square \mathcal{B}(A,V)$.
 
 Analogously, if we consider the diagram
  $$\xymatrix{  {}_\square \mathcal{B}(A,V)\otimes F \otimes A \ar[rr]^{ \id_{{}_\square \mathcal{B}(A,V)}\otimes u \otimes\id_A} \ar[rd]^{\sim} \ar[dd]^{\sim}
  & & {}_\square \mathcal{B}(A,V)\otimes {}_\square \mathcal{B}(A,V)\otimes A  \ar[dd]^{\id_{{}_\square \mathcal{B}(A,V)} \otimes \psi_{A,V}}  \ar[ld]^{\mu\otimes\id_A}\\
 & {}_\square \mathcal{B}(A,V)\otimes A \ar[dd]^(0.3){\psi_{A,V}} & \\
{}_\square \mathcal{B}(A,V) \otimes A \ar@{=}'[r][rr] \ar[rd]_{\psi_{A,V}}
  & & {}_\square \mathcal{B}(A,V)\otimes  A  \ar[ld]^{\psi_{A,V}}\\
 &  A  & 
 } $$
 we get the commutativity of the diagram
 $$\xymatrix{  {}_\square \mathcal{B}(A,V) \otimes F \ar[rr]^{\id_{{}_\square \mathcal{B}(A,V)} \otimes u} \ar[rd]^{\sim} 
  & & {}_\square \mathcal{B}(A,V)\otimes {}_\square \mathcal{B}(A,V) \ar[ld]^{\mu}\\
 & {}_\square \mathcal{B}(A,V)  & }$$
 Hence $1_{{}_\square \mathcal{B}(A,V)}$ is the right identity element of ${}_\square \mathcal{B}(A,V)$ and ${}_\square \mathcal{B}(A,V)$ is indeed a bialgebra.
 
 Suppose $B$ is another bialgebra and $\psi \colon B \otimes A \to A$ is an action
 such that $\cosupp \psi \subseteq V$.
 Denote by $\varphi \colon B \to {}_\square \mathcal{B}(A,V)$ the unique coalgebra homomorphism 
making the diagram \eqref{EqTheoremsquareBBialgebra} commutative.
We claim that $\varphi$ is a bialgebra homomorphism.

Consider the diagram

$$\xymatrix{  B \otimes B \otimes A 
\ar[rr]^(0.4){\varphi\otimes\varphi\otimes\id_A} \ar[dd]_{\id_B \otimes \psi}  \ar[rd]^{\mu_B\otimes\id_A} &&  {}_\square \mathcal{B}(A,V) \otimes {}_\square\mathcal{B}(A,V) \otimes A \ar'[d][dd]^{\id_{{}_\square \mathcal{B}(A,V)}\otimes \psi_{A,V}} \ar[rd]^{\mu\otimes\id_A} & \\
& B \otimes A \ar[rr]^(0.3){\varphi\otimes\id_A} \ar[dd]^(0.3){\psi} && {}_\square\mathcal{B}(A,V) \otimes A \ar[dd]^{\psi_{A,V}} \\
B \otimes A \ar'[r][rr]^(0.3){\varphi\otimes\id_A} \ar[rd]^{\psi} && {}_\square\mathcal{B}(A,V) \otimes A  \ar[rd]^{\psi_{A,V}} \\
& A \ar@{=}[rr] && A } $$ 
 where $\mu_B \colon B\otimes B \to B$ is the multiplication in $B$.
 The lower face and the front face are commutative by the definition of $\varphi$. The rear face is commutative since  it coincides with the lower face (and the front face) tensored  by $\varphi$ from the left.
 The left and the right faces are commutative since both $\psi$ and $\psi_{A,V}$ define on $A$ the structure
 of a left module. 
 
 Therefore, after the composition with $\psi_{A,V}$, the upper face becomes commutative too
 and the universal property of ${}_\square\mathcal{B}(A,V)$ implies the commutativity of the diagram
 $$\xymatrix{ B \otimes B   \ar[r]^(0.3){\varphi\otimes\varphi}\ar[d]_{\mu_B} & {}_\square \mathcal{B}(A,V) \otimes {}_\square\mathcal{B}(A,V)  \ar[d]^{\mu} \\ 
 B \ar[r]_{\varphi} & {}_\square\mathcal{B}(A,V) }$$ Therefore $\varphi$ preserves the multiplication.
 
 Consider the diagram $$\xymatrix{ & F \otimes A \ar[lddd]_{u_B \otimes \id_A} 
 \ar[dd]^(0.7){u \otimes \id_A}  \ar[rddd]^{\sim} & \\ & & \\
 & {}_\square\mathcal{B}(A,V) \otimes A  
 \ar[rd]_{\psi_{A,V}} & \\
 B \otimes A \ar[rr]_{\psi} \ar[ru]_{\varphi \otimes \id_A}& & A }$$ where $u_B \colon F  \to B$ is defined by $u(\alpha)=\alpha 1_B$
 for all $\alpha\in F$.
 
 The large triangle and the right triangle are commutative by the unitality of the module structures on $A$.
 The lower triangle is commutative by the definition of $\varphi$. Hence the left triangle becomes commutative
 after the composition with $\psi_{A,V}$. Thus the universal property of ${}_\square\mathcal{B}(A,V)$ implies the commutativity of the diagram
 $$\xymatrix{  & F \ar[ld]_{u_B} \ar[rd]^{u} & \\
 B \ar[rr]_{\varphi} & & {}_\square\mathcal{B}(A,V) }$$
 Therefore $\varphi$ preserves the identity element and is indeed a bialgebra homomorphism.
\end{proof}

\begin{example}
1) When $A$ is a unital associative algebra, the bialgebra ${}_\square \mathcal{B}(A,\End_F(A))$ is exactly the Sweedler universal measuring bialgebra~\cite[Chapter VII]{Sweedler} and we denote it simply by ${}_\square \mathcal{B}(A)$.

2) When $C$ is a coalgebra, the bialgebra ${}_\square \mathcal{B}(C,\End_F(C))$ will be called the \textit{universal acting bialgebra of $C$} and we denote it simply by ${}_{\sqbullet} \mathcal{B}(C)$.  
\end{example}

Consider now Hopf algebra actions.

 Recall that the embedding functor $\mathbf{Hopf}_F \to \mathbf{Bialg}_F$
has a right adjoint functor $H_r \colon \mathbf{Bialg}_F \to \mathbf{Hopf}_F$.
(See~\cite{AgoreCatConstr, CH}.) Let ${}_\square \mathcal{H}(A,V) := H_r({}_\square \mathcal{B}(A,V))$.
Define the action $\psi^{\mathbf{Hopf}}_{A,V} \colon {}_\square \mathcal{H}(A,V) \otimes A \to A$
 as the composition of $\psi_{A,V}$ and the counit ${}_\square \mathcal{H}(A,V)
\to {}_\square \mathcal{B}(A,V)$ of the adjunction tensored by $\id_A$.
Then for any Hopf algebra $H$ and any action $\psi \colon H \otimes A \to A$ with $\cosupp \psi \subseteq V$ there exists a unique Hopf algebra homomorphism $\varphi$
making the diagram below commutative:
$$\xymatrix{ H \otimes A \ar[r]^(0.6){\psi} \ar@{-->}[d]_{\varphi \otimes \id_A} & A \\
{}_\square \mathcal{H}(A,V) \otimes A  \ar[ru]_(0.6){\psi^{\mathbf{Hopf}}_{A,V}}  } $$ 
 We call ${}_\square \mathcal{H}(A,V)$ the \textit{$V$-universal acting Hopf algebra}.
 
\begin{example}\label{exSw} 
1) When $A$ is a unital associative algebra, the Hopf algebra ${}_\square \mathcal{H}(A,\End_F(A))$ will be called the \textit{universal acting Hopf algebra of $A$} and we denote it simply by ${}_\square \mathcal{H}(A)$.

2) When $C$ is a coalgebra, the Hopf algebra ${}_\square \mathcal{H}(C,\End_F(C))$ will be called the \textit{universal acting Hopf algebra of $C$} and we denote it simply by ${}_{\sqbullet} \mathcal{H}(C)$. It can be easily seen that ${}_{\sqbullet} \mathcal{H}(C)$ is the terminal object in the category whose objects are all module coalgebra structures $\theta_{H} \colon H\otimes C \to C$ on $C$ and the morphisms between two such objects  $\theta_{H_{1}} \colon H_{1}\otimes C \to C$ and $\theta_{H_{2}} \colon H_{2}\otimes C \to C$ are Hopf algebra homomorphisms $f \colon H_{1} \to H_{2}$ which make the following diagram commutative:
$$
\xymatrix{ H_{1} \otimes C\ar[r]^(0.6){\theta_{H_{1}}}\ar[d]_{f \otimes \mathrm{id}_{C}}   & C \\
H_{2} \otimes C \ar[ur]_{\theta_{H_{2}}} & {}}
$$

3) If $\psi \colon H \otimes A \to A$ is some Hopf algebra action on an ordinary algebra $A$,
 then ${}_\square \mathcal{H}(A,\cosupp \psi)$ is exactly the universal Hopf algebra of $\psi$
 introduced in~\cite{AGV1}. Indeed, by the definition of ${}_\square \mathcal{H}(A,\cosupp \psi)$
 we have $\cosupp\left(\psi^{\mathbf{Hopf}}_{A,\cosupp \psi}\right) \subseteq 
 \cosupp \psi$, the action   $\psi^{\mathbf{Hopf}}_{A,\cosupp \psi}$ is universal among all actions equivalent to $\psi$ and therefore $\cosupp \psi \subseteq \cosupp\left(\psi^{\mathbf{Hopf}}_{A,\cosupp \psi}\right)$. Hence $\cosupp\left(\psi^{\mathbf{Hopf}}_{A,\cosupp \psi}\right)=\cosupp \psi$.
 \end{example}
  
  Note that one can regard a unital algebra $A$ as a $\lbrace \mu \rbrace$-algebra
  and as a $\lbrace \mu, u \rbrace$-algebra. This leads to two different definitions
  of ${}_\square \mathcal{H}(A,V)$.
Theorem~\ref{TheoremHopfActionAlgebraUnital} below shows that when $V$ arises from some \textit{unital} action $\psi 
\colon H \otimes A \to A$
on $A$, i.e. $\psi(h \otimes 1_A)=\varepsilon(h) 1_A$ for all $h\in H$, then these two $V$-universal acting Hopf algebras are isomorphic.

\begin{theorem}\label{TheoremHopfActionAlgebraUnital}
Let $A$ be a unital algebra and let $V\subseteq \End_F(A)$
be a unital subalgebra.
Suppose $1_A$ is a common eigenvector for all operators from $V$. % $V 1_A = F 1_A$. 
(E.g., $V=\cosupp \psi$ for some unital action $\psi$.)
Then ${}_\square \mathcal{H}(A,V)$ does not depend on whether one regards $A$ as a $\lbrace \mu \rbrace$-algebra or as a $\lbrace \mu, u \rbrace$-algebra.
\end{theorem}
\begin{proof}
It is sufficient to prove that $\psi^\mathbf{Hopf}_{A,V} \colon {}_\square \mathcal{H}(A,V) \otimes A
\to A$ is a unital action even if we regard $A$ as a $\lbrace \mu \rbrace$-algebra.
The fact that $1_A$ is a common eigenvector for all operators from $V$ implies that
$1_A$ is a common eigenvector
for all operators from ${}_\square \mathcal{H}(A,V)$.
By~\cite[Proposition~4.3]{AGV1} the action $\psi^\mathbf{Hopf}_{A,V}$ is unital
and therefore ${}_\square \mathcal{H}(A,V)$ is a $V$-universal acting Hopf algebra on $A$ as a $\lbrace \mu, u \rbrace$-algebra too.
\end{proof}

\subsection{Coactions}\label{SectionCoactions}

Dually to the previous subsection, we can consider $\Omega$-algebras in the monoidal category of $B$-comodules over a bialgebra $B$, which leads to the following definition.

\begin{definition}
Let $B$ be a bialgebra and let $A$ be an $\Omega$-algebra. We say that 
a linear map $\rho \colon A \to A \otimes B$
is a \textit{$B$-coaction} on $A$ or, which is the same, $\rho$ defines on $A$ a structure of a \textit{(right) $B$-comodule $\Omega$-algebra} if the following two conditions hold: \begin{enumerate}
\item $\rho$ defines on $A$ a structure of a (counital) right $B$-comodule;
\item $\rho$ is a comeasuring from $A$ to $A$.
\end{enumerate}
\end{definition}

If $\Omega=\lbrace\mu, u\rbrace$ and $A$ is a unital algebra, then we get the usual definition of a unital comodule algebra.
If $\Omega=\lbrace\mu\rbrace$, then the algebra is not necessarily unital.
Finally, if $\Omega=\lbrace\Delta, \varepsilon\rbrace$ and $A$ is a coalgebra, then we get the definition of a comodule
coalgebra.

For coactions we can use the notion of cosupport, equivalence, and the preorder introduced in Section~\ref{SectionLinearMaps}.

Note that for a coaction $\rho \colon A \to A \otimes B$ its cosupport $\cosupp\rho$ is
just the image of the corresponding algebra homomorphism $B^* \to \End_F(A)$. 
In particular, $\cosupp\rho$ is a unital subalgebra in $\End_F(A)$ and $\supp\rho$ is a subcoalgebra in $B$.
Restricting linear functions on $B$ to $\supp\rho$, we obtain that $\cosupp\rho \cong (\supp\rho)^*$ as algebras by Lemma~\ref{lemmasup}. 

\begin{theorem}\label{TheoremBsquareBialgebra}
Let $A$ be an $\Omega$-algebra and let $V\subseteq \End_F(A)$ be a unital subalgebra closed in finite topology
such that there exists the $V$-universal comeasuring algebra $\mathcal{B}^\square(A,V) := \mathcal{A}^\square(A,A,V)$.
Then $\mathcal{B}^\square(A,V)$ admits a structure of a bialgebra such that for any bialgebra $B$ and any coaction $\rho \colon  A \to A \otimes B$ such that $\cosupp \rho \subseteq V$ the unique algebra homomorphism $\varphi$
making the diagram below commutative is in fact a bialgebra homomorphism:
\begin{equation}\label{EqTheoremBsquareBialgebra}\xymatrix{ A \ar[r]^(0.3){\rho_{A,V}} \ar[rd]_\rho & A \otimes \mathcal{B}^\square(A,V) \ar@{-->}[d]^{\id_A\otimes
\varphi} \\
& A\otimes B} \end{equation}
(Here $\rho_{A,V} := \rho_{A,A,V}$.)
\end{theorem}

\begin{proof}
Note that $(\rho_{A,V}\otimes \id_{\mathcal{B}^\square(A,V)})\rho_{A,V}
\colon A \to A\otimes \mathcal{B}^\square(A,V) \otimes \mathcal{B}^\square(A,V)$
is a comeasuring too and, since $V$ is a subalgebra closed in finite topology, we have $\cosupp((\rho_{A,V}\otimes \id_{\mathcal{B}^\square(A,V)})\rho_{A,V})\subseteq V$.
Therefore, there exists a unique algebra homomorphism 
$\Delta \colon \mathcal{B}^\square(A,V) \to \mathcal{B}^\square(A,V) \otimes \mathcal{B}^\square(A,V)$
making the diagram below commutative:
$$\xymatrix{ A \ar[rr]^{\rho_{A,V}} \ar[d]_{\rho_{A,V}} && A\otimes \mathcal{B}^\square(A,V)
\ar@{-->}[d]^{\id_A\otimes \Delta} \\
A\otimes \mathcal{B}^\square(A,V) \ar[rr]_(0.4){\rho_{A,V}\otimes \id_{\mathcal{B}^\square(A,V)}} &&  A\otimes \mathcal{B}^\square(A,V) \otimes \mathcal{B}^\square(A,V)}$$
Now we define the comultiplication in $\mathcal{B}^\square(A,V)$ to be equal to $\Delta$.

In order to define the counit map, we consider the trivial map $$A
\mathrel{\widetilde\to} A  \otimes F,$$ which is a comeasuring.
There exists a unique algebra homomorphism $\varepsilon \colon   \mathcal{B}^\square(A,V) \to F$
making the diagram below commutative:
$$\xymatrix{ A \ar[r]^(0.3){\rho_{A,V}} \ar[rd]^{\sim} & A \otimes \mathcal{B}^\square(A,V)
\ar[d]^{\id_A \otimes \varepsilon}\\
& A \otimes F}$$
We define the counit in $\mathcal{B}^\square(A,V)$ to be equal to $\varepsilon$.

Now we have to prove that the comultiplication $\Delta$ is coassociative and that $\varepsilon$
and $\Delta$ indeed satisfy the counit axioms.

Consider the diagram

{\tiny
$$\xymatrix{ A \ar[rr]^{\rho_{A,V}} \ar[dd]_{\rho_{A,V}} \ar[rd]^{\rho_{A,V}} && A \otimes \mathcal{B}^\square(A,V) \ar[rd]^{\qquad \rho_{A,V} \otimes \id_{\mathcal{B}^\square(A,V)}} \ar'[d][dd]_{\rho_{A,V} \otimes \id_{\mathcal{B}^\square(A,V)}} & \\
& A \otimes \mathcal{B}^\square(A,V) \ar[rr]^(0.3){\id_A \otimes \Delta} \ar[dd]_(0.3){\rho_{A,V} \otimes \id_{\mathcal{B}^\square(A,V)}} && A \otimes \mathcal{B}^\square(A,V) \otimes \mathcal{B}^\square(A,V)\ar[dd]_(0.3){\rho_{A,V} \otimes \id_{\mathcal{B}^\square(A,V) \otimes \mathcal{B}^\square(A,V)}} \\
A \otimes \mathcal{B}^\square(A,V)  \ar'[r][rr]_(0.3){\id_A \otimes \Delta}
 \ar[rd]^{\id_A \otimes \Delta} && A \otimes \mathcal{B}^\square(A,V) \otimes \mathcal{B}^\square(A,V) \ar[rd]_(0.35){\id_A \otimes \Delta \otimes \id_{\mathcal{B}^\square(A,V)}\qquad} &\\
 & A \otimes \mathcal{B}^\square(A,V) \otimes \mathcal{B}^\square(A,V) 
 \ar[rr]_(0.4){\id_A  \otimes \id_{\mathcal{B}^\square(A,V)}\otimes \Delta}  && A \otimes \mathcal{B}^\square(A,V) \otimes \mathcal{B}^\square(A,V) \otimes \mathcal{B}^\square(A,V) } $$}

The right face is commutative since
it is the definition of $\Delta$ tensored by $\id_{\mathcal{B}^\square(A,V)}$.
The commutativity of the left, the upper and the rear face follows from
the definition of $\Delta$ too.
The front face is commutative since both compositions are equal to $\rho_{A,V} \otimes \Delta$.
Therefore the compositions of the lower face being composed with $\rho_{A,V}$
are equal too.
Now the universal property of $\mathcal{B}^\square(A,V)$ implies the commutativity of the diagram
 $$\xymatrix{ \mathcal{B}^\square(A,V) \ar[rr]_\Delta\ar[d]^\Delta && \mathcal{B}^\square(A,V) \otimes \mathcal{B}^\square(A,V) \ar[d]^{\Delta\otimes \id_{\mathcal{B}^\square(A,V)}}\\ 
 \mathcal{B}^\square(A,V) \otimes \mathcal{B}^\square(A,V) \ar[rr]_(0.42){ \id_{\mathcal{B}^\square(A,V)}\otimes\Delta} && \mathcal{B}^\square(A,V) \otimes \mathcal{B}^\square(A,V) \otimes \mathcal{B}^\square(A,V) }$$
 and the coassociativity follows.

Consider the diagram
$$\xymatrix{ A \ar[rr]^{\rho_{A,V}} \ar[rd]^{\rho_{A,V}} \ar[dd]_{\rho_{A,V}} && A \otimes \mathcal{B}^\square(A,V) \ar@{=}[ld]\ar[dd]^{\rho_{A,V} \otimes \id_{\mathcal{B}^\square(A,V)}} \\
& A \otimes \mathcal{B}^\square(A,V) \ar[dd]^(0.3){\sim} & \\
A \otimes \mathcal{B}^\square(A,V)\ar[rd]^{\sim} \ar'[r][rr]^(0.3){\id_A\otimes \Delta} && A \otimes \mathcal{B}^\square(A,V) \otimes \mathcal{B}^\square(A,V) \ar[ld]^{\qquad\id_A\otimes\varepsilon\otimes\id_{\mathcal{B}^\square(A,V)}} \\
& A \otimes F \otimes \mathcal{B}^\square(A,V) & } $$

 The commutativity of the upper face and the left face is obvious. The commutativity of the right face
 follows from the definition of $\varepsilon$ and the commutativity of the rear face follows from the definition
 of $\Delta$. Therefore, the algebra homomorphisms forming the lower face become equal after their
 composition with $\rho_{A,V}$. 
 Now universal property of $\mathcal{B}^\square(A,V)$ implies the commutativity of the diagram
 $$\xymatrix{ \mathcal{B}^\square(A,V) \ar[rr]^{\Delta} \ar[rd]^{\sim} 
  & & \mathcal{B}^\square(A,V)\otimes \mathcal{B}^\square(A,V) \ar[ld]^{\qquad\varepsilon\otimes \id_{\mathcal{B}^\square(A,V)}}\\
 & F \otimes \mathcal{B}^\square(A,V)  & }$$
 Hence $\varepsilon$ satisfies the left counit axiom.
 
 Analogously, if we consider the diagram
  $$\xymatrix{ A \ar[rr]^{\rho_{A,V}}\ar[dd]_{\rho_{A,V}} \ar[rd]^{\sim} && A\otimes \mathcal{B}^\square(A,V)
  \ar[ld]^{\qquad\id_A\otimes \varepsilon} \ar[dd]^{\rho_{A,V} \otimes \id_{\mathcal{B}^\square(A,V)}} \\
  & A\otimes F \ar[dd]^(0.3){\rho_{A,V} \otimes \id_F} & \\
 A \otimes \mathcal{B}^\square(A,V)\ar[rd]^{\sim} \ar'[r][rr]_(0.3){\id_A\otimes \Delta} && A \otimes \mathcal{B}^\square(A,V) \otimes \mathcal{B}^\square(A,V) \ar[ld]^{\qquad\id_A\otimes\id_{\mathcal{B}^\square(A,V)}\otimes\varepsilon} \\
& A \otimes \mathcal{B}^\square(A,V) \otimes F & } $$
  we get the commutativity of the diagram
 $$\xymatrix{ \mathcal{B}^\square(A,V) \ar[rr]^{\Delta} \ar[rd]^{\sim} 
  & & \mathcal{B}^\square(A,V)\otimes \mathcal{B}^\square(A,V) \ar[ld]^{\qquad \id_{\mathcal{B}^\square(A,V)}\otimes\varepsilon}\\
 & \mathcal{B}^\square(A,V) \otimes  F & }$$
 Hence $\varepsilon$ satisfies the right counit axiom and $\mathcal{B}^\square(A,V)$ is indeed a bialgebra.
 
 Suppose $B$ is another bialgebra and $\rho \colon  A \to A\otimes B$ is a coaction
 such that $\cosupp \rho \subseteq V$.
 Denote by $\varphi \colon \mathcal{B}^\square(A,V) \to B$ the unique algebra homomorphism 
making the diagram~\eqref{EqTheoremBsquareBialgebra} commutative.
We claim that $\varphi$ is a bialgebra homomorphism.

Consider the diagram
$$\xymatrix{  A \ar@{=}[rr]\ar[dd]_{\rho_{A,V}}\ar[rd]^{\rho_{A,V}} && A \ar'[d][dd]_{\rho}  \ar[rd]^{\rho}& \\
& A \otimes \mathcal{B}^\square(A,V) \ar[rr]^(0.4){\id_A\otimes\varphi}\ar[dd]_(0.3){\rho_{A,V} \otimes \id_{\mathcal{B}^\square(A,V)}} && A\otimes B \ar[dd]^{\rho \otimes \id_{B}} \\
A \otimes \mathcal{B}^\square(A,V) \ar'[r][rr]^(0.3){\id_A\otimes\varphi}\ar[rd]^{\id_A \otimes \Delta} && A\otimes B \ar[rd]^{\id_A \otimes \Delta_B}& \\
& A \otimes \mathcal{B}^\square(A,V)\otimes \mathcal{B}^\square(A,V) \ar[rr]^{\id_A\otimes\varphi\otimes\varphi} & & A\otimes B\otimes B } $$
 where $\Delta_B \colon B \to B\otimes B$ is the comultiplication in $B$.
 The upper face and the rear face are commutative by the definition of $\varphi$. The front face is commutative since it coincides with the upper face (and the rear face) tensored  by $\varphi$ from the right.
 The left and the right faces are commutative since both $\rho$ and $\rho_{A,V}$ define on $A$ the structure
 of a right comodule. 
 
 Therefore, after the composition with $\rho_{A,V}$, the lower face becomes commutative too
 and the universal property of $\mathcal{B}^\square(A,V)$ implies the commutativity of the diagram
 $$\xymatrix{\mathcal{B}^\square(A,V)\ar[d]_{\Delta}  \ar[r]^{\varphi} & B \ar[d]^{\Delta_B}\\
   \mathcal{B}^\square(A,V) \otimes \mathcal{B}^\square(A,V)  \ar[r]_(0.7){\varphi\otimes\varphi} & B \otimes B   \\ 
  }$$ Therefore $\varphi$ preserves the comultiplication.
 
 Consider the diagram $$\xymatrix{ & A \ar[lddd]_{\rho} 
 \ar[dd]^(0.7){\rho_{A,V}}  \ar[rddd]^{\sim} & \\ & & \\
 &  A  \otimes \mathcal{B}^\square(A,V) 
 \ar[rd]_{\id_A\otimes\varepsilon\ } \ar[ld]^{\ \id_A \otimes \varphi}& \\
 A \otimes B \ar[rr]_{\id_A\otimes\varepsilon_B} & & A\otimes F }$$ where $\varepsilon_B$
 is the counit in $B$.
 
 The large triangle and the right triangle are commutative by the counitality of the comodule structures on $A$.
 The left triangle is commutative by the definition of $\varphi$. Hence the lower triangle becomes commutative
 after the composition with $\rho_{A,V}$. Thus the universal property of $\mathcal{B}^\square(A,V)$ implies the commutativity of the diagram
 $$\xymatrix{ \mathcal{B}^\square(A,V)\ar[rr]^{\varphi}\ar[rd]_{\varepsilon} & & B \ar[ld]^{\varepsilon_B}  \\
  & F   &}$$
 Therefore $\varphi$ preserves the counit and is indeed a bialgebra homomorphism.
\end{proof}

We call $\mathcal{B}^\square(A,V)$ the \textit{$V$-universal coacting bialgebra}.
\begin{example}\label{exUniv1}
1) Let $A$ be a unital associative algebra for which the bialgebra $\mathcal{B}^\square(A,\End_F(A))$ exists. It can be easily seen that  $\mathcal{B}^\square(A,\End_F(A))$ is exactly the Tambara universal coacting bialgebra $\alpha(A,A)$ (see~\cite{Tambara}) and we will denote it simply by $\mathcal{B}^\square(A)$. Note that if $A$ is finite dimensional, the conditions in Theorem~\ref{TheoremUnivComeasExistence} which ensure the existence of $\mathcal{B}^\square(A)$ are trivially fulfilled. 

2) Let $A=\bigoplus_{n\in \mathbb Z} A^{(n)}$ be an associative $\mathbb Z$-graded 
unital algebra such that $\dim A^{(n)} < +\infty$ for all $n\in\mathbb Z$.
Denote by $V\subseteq \End_F(A)$ 
the subalgebra of all linear operators preserving the grading.
Obviously, $V$
is pointwise finite dimensional. 
Moreover, if an operator $f_0\in\End_F(A)$ 
does not preserve the grading, for some $n\in\mathbb Z$, $a\in A^{(n)}$ and $b\notin A^{(n)}$
we have $f_0(a)=b$.
Note that the neighbourhood $\lbrace f\in \End_F(A) \mid f(a)=b \rbrace$ of the point $f_0$
in the finite topology has empty intersection with $V$.
Hence $V$ is closed in the finite topology and Theorem~\ref{TheoremUnivComeasExistence}
implies that there exists the bialgebra $\mathcal{B}^\square(A, V)$,
which is Yu.\,I.~Manin's universal coacting bialgebra $\underline{\mathrm{end}}(A)$~\cite{Manin}.

3) Let $C$ be a coalgebra for which the bialgebra $\mathcal{B}^\square(C,\End_F(C))$ exists. In this case we denote the bialgebra $\mathcal{B}^\square(C,\End_F(C))$ simply by  $\mathcal{B}^{\sqbullet}(C)$ and we call it the \textit{universal coacting bialgebra of $C$}. As in the previous example, if $C$ is finite dimensional the conditions in Theorem~\ref{TheoremUnivComeasExistence} which ensure the existence of $\mathcal{B}^{\sqbullet}(C)$ are trivially fulfilled.  
\end{example}

Consider now Hopf algebra coactions.

 Recall that the embedding functor $\mathbf{Hopf}_F \to \mathbf{Bialg}_F$
has a left adjoint functor $H_l \colon \mathbf{Bialg}_F \to \mathbf{Hopf}_F$.
(See~\cite[Theorem 2.6.3]{Pareigis} or~\cite{Takeuchi}.) Let $\mathcal{H}^\square(A,V) := H_l(\mathcal{B}^\square(A,V))$.
Define the coaction $\rho^{\mathbf{Hopf}}_{A,V} \colon A \to A \otimes \mathcal{H}^\square(A,V)$
 as the composition of $\rho_{A,V}$ and the unit $\mathcal{B}^\square(A,V) \to \mathcal{H}^\square(A,V)$ of the adjunction tensored by $\id_A$.
Then for any Hopf algebra $H$ and any coaction $\rho \colon A \to A \otimes H$ with $\cosupp \rho \subseteq V$ there exists the unique Hopf algebra homomorphism $\varphi$
making the diagram below commutative:
$$\xymatrix{ A \ar[rd]_\rho \ar[r]^(0.3){\rho^{\mathbf{Hopf}}_{A,V}} & A \otimes \mathcal{H}^\square(A,V)
\ar@{-->}[d]^{\id_A \otimes \varphi} \\
& A \otimes H }$$
 We call $\mathcal{H}^\square(A,V)$ the \textit{$V$-universal coacting Hopf algebra}.
 
\begin{example}\label{exUniv2}
1) If $A$ is a unital associative algebra for which the bialgebra $\mathcal{B}^\square(A)$ exists, the Hopf algebra $\mathcal{H}^\square(A,\End_F(A))$ is exactly the Tambara universal coacting Hopf algebra and we denote it by $\mathcal{H}^\square(A)$.

2) Let $A=\bigoplus_{n\in \mathbb Z} A^{(n)}$ be an associative $\mathbb Z$-graded 
unital algebra such that $\dim A^{(n)} < +\infty$ for all $n\in\mathbb Z$.
Denote by $V\subseteq \End_F(A)$ 
the subalgebra of all linear operators preserving the grading.
As we noted in Example~\ref{exUniv1}, 2),
the space $V$ is pointwise finite dimensional and closed in the finite topology. Hence Theorem~\ref{TheoremUnivComeasExistence}
implies that there exists the Hopf algebra $\mathcal{H}^\square(A, V)$,
which is Yu.\,I.~Manin's universal coacting Hopf algebra $\underline{\mathrm{aut}}(A)$~\cite{Manin}. 

3) Let $C$ be a coalgebra for which the bialgebra $\mathcal{B}^{\sqbullet}(C)$ exists. In this case we denote the Hopf algebra $\mathcal{H}^\square(C,\End_F(C))$ simply by $\mathcal{H}^{\sqbullet}(C)$ and we call it the \textit{universal coacting Hopf algebra of $C$}. It is straightforward to see that  $\mathcal{H}^{\sqbullet}(C)$ is the initial object in the category whose objects are all comodule coalgebra structures $\psi_{H} \colon C \to C \otimes H$ and the morphisms between two such objects $\psi_{H_{1}} \colon C \to C \otimes H_{1}$ and $\psi_{H_{2}} \colon C \to C \otimes H_{2}$ are Hopf algebra homomorphisms $f \colon H_{1} \to H_{2}$ which make the following diagram commutative:
$$
\xymatrix{ C \ar[rr]^-{\rho_{H_{1}}}\ar[rrd]_{\rho_{H_{2}}}   & {} & C \otimes H_{1}\ar[d]^{\mathrm{id}_{C} \otimes  f}\\
{} & {} & {C \otimes H_{2}}}
$$
 
4) If $\rho \colon  A \to A \otimes H$ is some Hopf algebra coaction on an ordinary algebra $A$,
 then $\mathcal{H}^\square(A,\cosupp \rho)$ is exactly the universal Hopf algebra of $\rho$
 introduced in~\cite{AGV1}. Indeed, by the definition of $\mathcal{H}^\square(A,\cosupp \rho)$
 we have $\cosupp\left(\rho^{\mathbf{Hopf}}_{A,\cosupp \rho}\right) \subseteq 
 \cosupp \rho$, the coaction   $\rho^{\mathbf{Hopf}}_{A,\cosupp \rho}$ is universal among all coactions equivalent to $\rho$ and therefore $\cosupp \rho\subseteq\cosupp \left(\rho^{\mathbf{Hopf}}_{A,\cosupp \rho}\right)$. Hence $\cosupp\left(\rho^{\mathbf{Hopf}}_{A,\cosupp \rho}\right)=\cosupp \rho$.
 \end{example}
 
 \begin{theorem}\label{TheoremBHsquareVExistence}
Let $A$ be an $\Omega$-algebra and let $V\subseteq \End_F(A)$ be a unital subalgebra
closed in the finite topology.
Then $\mathcal{B}^\square(A,V)$ exists if and only if the subalgebra generated by $\cosupp\rho$
for all bialgebras $B$ and coactions $\rho \colon  A \to A\otimes B$
 such that $\cosupp \rho \subseteq V$, is pointwise finite dimensional.
 Analogously,  $\mathcal{H}^\square(A,V)$ exists if and only if the subalgebra generated by $\cosupp\rho$
for all Hopf algebras $H$ and coactions $\rho \colon  A \to A\otimes H$
 such that $\cosupp \rho \subseteq V$, is pointwise finite dimensional.
 \end{theorem}
\begin{proof} The proof is identical for both cases. Therefore without loss of generality we may treat just the case of bialgebras.

The necessity is obvious as if $\mathcal{B}^\square(A,V)$ indeed exists, we have $\cosupp \rho \subseteq \cosupp \rho_{A,V}$ for all bialgebras $B$ and coactions $\rho \colon  A \to A\otimes B$
 such that $\cosupp \rho \subseteq V$.

Let $W \subseteq \End_F(A)$ be the subalgebra generated by all such $\cosupp\rho\subseteq V$.
If $W$ is pointwise finite dimensional, then by Lemma~\ref{LemmaClosurePwFD}
its closure $\overline W$ in the finite topology is pointwise finite dimensional too
and by Theorem~\ref{TheoremUnivComeasExistence} there exists $\mathcal{B}^\square(A,\overline W)$.
As $\overline W \subseteq V$, we have $\mathcal{B}^\square(A,V)=\mathcal{B}^\square(A,\overline W)$.
\end{proof} 
\begin{corollary}\label{CorollaryManinExistence}
The Tambara universal coacting Hopf algebra $\mathcal{H}^\square(A,\End_F(A))$ exists for an $\Omega$-algebra $A$ if
and only if the subalgebra generated by $\cosupp\rho$
for all Hopf algebras $H$ and all coactions $\rho \colon  A \to A\otimes H$ is pointwise finite dimensional.
\end{corollary} 

Proposition~\ref{PropositionHopfCoactionAlgebraUnital} below is dual 
to~\cite[Proposition~4.3]{AGV1}.

\begin{proposition}\label{PropositionHopfCoactionAlgebraUnital}
Let $A$ be a unital algebra and let $\rho \colon A \to A \otimes H$
be a coaction of a Hopf algebra $H$.
Suppose $\rho(1_A)=1_A \otimes h$ for some $h\in H$.
Then $h=1_H$, i.e. the coaction $\rho$ is \textit{unital}.
\end{proposition}
\begin{proof} Since $\rho$ is a coaction, we have $\varepsilon(h)=1_F$ and
$\Delta(h) = h \otimes h$. Considering $1_A^2=1_A$, we get
$h^2=h$. Since $(Sh)h=h(Sh)=1_H$, the element $h$ is invertible.
Now $h^2=h$ implies $h=1_H$.
\end{proof}

Theorem~\ref{TheoremHopfCoactionAlgebraUnital} below shows that when $V$ arises from some unital coaction,
there is no difference whether we include the unit map to $\Omega$ or not.

\begin{theorem}\label{TheoremHopfCoactionAlgebraUnital}
Let $A$ be a unital algebra and let $V\subseteq \End_F(A)$
be a unital subalgebra such that there exists $\mathcal{H}^\square(A,V)$.
Suppose $V 1_A = F 1_A$. (E.g., $V=\cosupp \rho$ for some unital coaction $\rho$.)
Then $\mathcal{H}^\square(A,V)$ does not depend on whether one regards $A$ as a $\lbrace \mu \rbrace$-algebra or as a $\lbrace \mu, u \rbrace$-algebra.
\end{theorem}
\begin{proof} First regard $A$ as a $\lbrace \mu \rbrace$-algebra.
The equality $V 1_A = F 1_A$ implies 
$\rho^\mathbf{Hopf}_{A,V}(1_A)=1_A \otimes h$ for some $h\in \mathcal{H}^\square(A,V)$.
By Proposition~\ref{PropositionHopfCoactionAlgebraUnital}, $\rho^\mathbf{Hopf}_{A,V}$ is a unital
coaction and $\mathcal{H}^\square(A,V)$ is a $V$-universal coacting Hopf algebra on $A$ as a $\lbrace \mu, u \rbrace$-algebra too.
\end{proof}

\subsection{Duality between actions and coactions}\label{SectionDualityActionsCoactions}

It turns out that in the case of (co)actions the coalgebra isomorphism $\theta$
from Theorem~\ref{TheoremUniv(Co)measDuality} is in fact a bialgebra isomorphism.

\begin{theorem}\label{TheoremUniv(Co)actDuality}
Let $A$ be an $\Omega$-algebra and let $V\subseteq \End_F(A)$ be a unital pointwise finite dimensional subalgebra  closed in the finite topology.
If $\theta$ is the unique bialgebra homomorphism making
the diagram
$$\xymatrix{ {}_\square \mathcal{B}(A,V) \otimes A \ar[dr]^(0.6){\psi_{A,V}} & \\
                   &   A \\
\mathcal{B}^\square(A,V)^\circ \otimes A \ar[ru]_(0.6){\widetilde{\rho_{A,V}}}\ar@{-->}[uu]^{\theta \otimes \id_A} } $$  commutative,
then $\theta$ is a bialgebra isomorphism. In particular, for any finite dimensional algebra $A$ and any finite dimensional coalgebra $C$ we have the following bialgebra isomorphisms:
$$
{}_{\square}\mathcal{B}(A) \cong \mathcal{B}^{\square}(A)^{\circ}, \qquad{}_{\sqbullet}\mathcal{B}(C) \cong \mathcal{B}^{\sqbullet}(C)^{\circ}.
$$
\end{theorem}
\begin{proof}
The existence of $\theta$ follows from Theorem~\ref{TheoremsquareBBialgebra}. By
Theorem~\ref{TheoremUniv(Co)measDuality} $\theta$ is bijective.
\end{proof}

The same is true for $V$-universal Hopf algebras.

\begin{theorem}\label{TheoremUnivHopf(Co)actDuality}
Let $A$ be an $\Omega$-algebra and let $V\subseteq \End_F(A)$ be a unital pointwise finite dimensional subalgebra  closed in the finite topology.
Then the unique homomorphism $\theta^\mathbf{Hopf} \colon \mathcal{H}^\square(A,V)^\circ \to {}_\square \mathcal{H}(A,V) $ of Hopf algebras  making
the diagram $$\xymatrix{ {}_\square \mathcal{H}(A,V) \otimes A \ar[dr]^(0.6){\psi^\mathbf{Hopf}_{A,V}} & \\
                   &   A \\
\mathcal{H}^\square(A,V)^\circ \otimes A \ar[ru]_(0.6){\widetilde{\rho^\mathbf{Hopf}_{A,V}}}\ar@{-->}[uu]^{\theta^\mathbf{Hopf} \otimes \id_A} } $$  commutative is an isomorphism. In particular, for any finite dimensional algebra $A$ and any finite dimensional coalgebra $C$ we have the following Hopf algebra isomorphisms:
$$
{}_{\square}\mathcal{H}(A) \cong \mathcal{H}^{\square}(A)^{\circ}, \qquad{}_{\sqbullet}\mathcal{H}(C) \cong \mathcal{H}^{\sqbullet}(C)^{\circ}.
$$
\end{theorem}
\begin{proof}Let $H$ be an arbitrary Hopf algebra.
Consider the natural bijections
\begin{equation*}
\begin{split}
\mathbf{Hopf}_F(H, {}_\square \mathcal{H}(A,V)) = 
\mathbf{Hopf}_F(H, H_r({}_\square \mathcal{B}(A,V)))\cong
\mathbf{Bialg}_F(H, {}_\square \mathcal{B}(A,V)) \\
\cong \mathbf{Bialg}_F(H, \mathcal{B}^\square(A,V)^\circ)
\cong \mathbf{Bialg}_F(\mathcal{B}^\square(A,V), H^\circ)
\cong \mathbf{Hopf}_F(H_l(\mathcal{B}^\square(A,V)), H^\circ) \\
= \mathbf{Hopf}_F(\mathcal{H}^\square(A,V), H^\circ)
\cong \mathbf{Hopf}_F(H, \mathcal{H}^\square(A,V)^\circ).\end{split}\end{equation*}
Hence $\mathcal{H}^\square(A,V)^\circ \cong {}_\square \mathcal{H}(A,V)$
under the isomorphism that corresponds to $$\id_{\mathcal{H}^\square(A,V)^\circ}
\in \mathbf{Hopf}_F(\mathcal{H}^\square(A,V)^\circ, \mathcal{H}^\square(A,V)^\circ)$$
if we take $H=\mathcal{H}^\square(A,V)^\circ$.
But the corresponding element of $\mathbf{Hopf}_F(\mathcal{H}^\square(A,V)^\circ, {}_\square \mathcal{H}(A,V))$ is precisely $\theta^\mathbf{Hopf}$ resulting from the universal properties 
of ${}_\square \mathcal{B}(A,V)$ and ${}_\square \mathcal{H}(A,V)$:
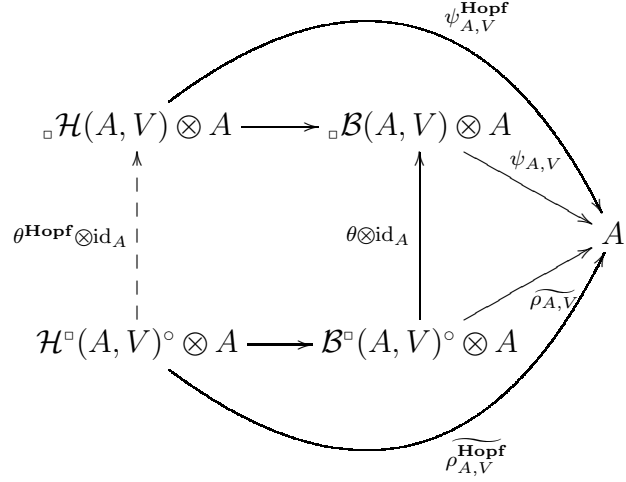
$$\xymatrix{ {}_\square \mathcal{H}(A,V) \otimes A \ar[r] \ar@/^5pc/[rrd]^(0.6){\psi^\mathbf{Hopf}_{A,V}} &
{}_\square \mathcal{B}(A,V) \otimes A \ar[rd]^{\psi_{A,V}} & \\
                 &  &   A \\
\mathcal{H}^\square(A,V)^\circ \otimes A \ar[r] \ar@/_5pc/[rru]_(0.6){\widetilde{\rho^\mathbf{Hopf}_{A,V}}}\ar@{-->}[uu]^{\theta^\mathbf{Hopf} \otimes \id_A} &  \mathcal{B}^\square(A,V)^\circ \otimes A
\ar[uu]^{\theta \otimes \id_A}
 \ar[ru]_(0.6){\widetilde{\rho_{A,V}}} & }$$
The uniqueness of $\theta^\mathbf{Hopf}$ follows from the universal property of ${}_\square \mathcal{H}(A,V)$ too.
\end{proof}

\subsection{Duality between $\Omega$- and $\Omega^*$-algebras}\label{SectionDualityAlgCoalg}

Here we show how the universal (co)acting bialgebras and Hopf algebras relate to each other if we replace a finite dimensional $\Omega$-algebra $A$ with the dual $\Omega^*$-algebra $A^*$. 

\begin{theorem}\label{dual1} Let $A$ be a finite dimensional $\Omega$-algebra and let $V \subseteq \End_F(A^*)$ be a unital subalgebra. Then we have a bialgebra isomorphism $${}_\square \mathcal{B}(A^*,V)
\cong {}_\square \mathcal{B}(A,V^{\#})^{\mathrm{op}}$$
where $V^{\#} := \lbrace f^* \mid f \in V \rbrace\subseteq \End_F(A)$.
\end{theorem}
\begin{proof}
Consider the category $\mathbf{Actions}(A^*,V)$ where the objects are all actions $\psi \colon B \otimes A^* \to A^*$ for all bialgebras $B$ with $\cosupp \psi \subseteq V$ and the morphisms from $\psi_1 \colon B_1 \otimes A^* \to A^*$
to $\psi_2 \colon B_2 \otimes A^* \to A^*$  are bialgebra homomorphisms $\varphi \colon B_1 \to B_2$
making the diagram below commutative:

$$\xymatrix{ B_1 \otimes A^* \ar[r]^(0.6){\psi_1} \ar[d]_{\varphi \otimes \id_{A^*}} & A^* \\
B_2 \otimes A^*  \ar[ru]_{\psi_2}  } $$

Then the terminal object of $\mathbf{Actions}(A^*,V)$ is the action
of the bialgebra ${}_\square \mathcal{B}(A^*,V)$ on $A^*$.

Analogously, the action of the bialgebra ${}_\square \mathcal{B}(A,V^{\#})$ is the terminal object in the category $\mathbf{Actions}(A,V^{\#})$
where the objects are all actions $\psi \colon B \otimes A \to A$ for all bialgebras $B$ with $\cosupp \psi \subseteq V^{\#}$ and the morphisms from $\psi_1 \colon B_1 \otimes A \to A$
to $\psi_2 \colon B_2 \otimes A \to A$  are bialgebra homomorphisms $\varphi \colon B_1 \to B_2$
making the diagram below commutative:

$$\xymatrix{ B_1 \otimes A \ar[r]^(0.6){\psi_1} \ar[d]_{\varphi \otimes \id_A} & A \\
B_2 \otimes A  \ar[ru]_{\psi_2}  } $$

Let $\widehat{\ }$ be the functor $\mathbf{Actions}(A^*,V) \to \mathbf{Actions}(A,V^{\#})$
defined as follows: if $\psi \colon B \otimes A^* \to A^*$ is an object of $\mathbf{Actions}(A^*,V)$,
then $\widehat\psi \colon B^\mathrm{op} \otimes A \to A$ is defined by
$$a^*\left(\widehat\psi(b\otimes a)\right):=\psi(b\otimes a^*)(a)\text{ for all }b\in B,\ a\in A,\ a^*\in A^*.$$
In other words, $\widehat\psi(b\otimes(-)) := (\psi(b\otimes(-)))^*$ and $\cosupp \widehat\psi = (\cosupp\psi)^{\#}$. The $\Omega$-algebra $A$ becomes
a right $B$-module and therefore a left $B^\mathrm{op}$-module.

A straightforward check shows that $\widehat{\ }$ is an isomorphism of categories and therefore 
$\widehat{\ }$ maps the terminal object to the terminal object. Now the theorem follows.
\end{proof}

For a certain class of finite dimensional $\Omega$-algebras, the bialgebra isomorphism from the previous theorem can be extended to the corresponding $V$-universal (co)acting Hopf algebras. 

\begin{theorem}\label{dual3}
Let $A$ be a finite dimensional $\Omega$-algebra and let $V \subseteq \End_F(A^*)$ be a unital subalgebra such that ${}_\square \mathcal{B}(A,V^{\#})$ is a Hopf algebra. Then we have a Hopf algebra isomorphism $${}_\square \mathcal{H}(A^*,V)
\cong {}_\square \mathcal{H}(A,V^{\#})^{\mathrm{op}}.$$
\end{theorem}
\begin{proof}
To start with, according to our assumption ${}_\square \mathcal{B}(A,V^{\#})^{\mathrm{op}}$ has a skew antipode. Proceeding as in \cite[Lemma 3.6]{CH}, one can easily check that $H_{r}\bigl({}_\square \mathcal{B}(A,V^{\#})^{\mathrm{op}}\bigl)$ is a Hopf algebra with bijective antipode. Since the embedding functor $\mathbf{Hopf}_F \to \mathbf{Bialg}_F$ is left adjoint to $H_{r}$, for any Hopf algebra with bijective antipode $H$ we have the following natural bijections:
\begin{eqnarray*}
&&\Hom_{\,\mathbf{SHopf}_{F}}\bigl(H,\, H_{r}({}_\square \mathcal{B}(A,V^{\#})^{\mathrm{op}})\bigl) = \Hom_{\,\mathbf{Hopf}_{F}}\bigl(H,\, H_{r}({}_\square \mathcal{B}(A,V^{\#})^{\mathrm{op}})\bigl)\\
&&\cong \Hom_{\,\mathbf{Bialg}_{F}} \bigl(H,\, {}_\square \mathcal{B}(A,V^{\#})^{\mathrm{op}}\bigl) = \Hom_{\,\mathbf{Bialg}_{F}} \bigl(H^{\mathrm{op}},\, {}_\square \mathcal{B}(A,V^{\#})\bigl)\\
&& \cong \Hom_{\,\mathbf{Hopf}_{F}} \bigl(H^{\mathrm{op}},\, H_{r}({}_\square \mathcal{B}(A,V^{\#}))\bigl) = \Hom_{\,\mathbf{Hopf}_{F}} \bigl(H,\, H_{r}({}_\square \mathcal{B}(A,V^{\#}))^{\mathrm{op}}\bigl)\\
&& = \Hom_{\,\mathbf{SHopf}_{F}} \bigl(H,\, H_{r}({}_\square \mathcal{B}(A,V^{\#}))^{\mathrm{op}}\bigl).
\end{eqnarray*}
Thus, the Hopf algebras $H_{r}({}_\square \mathcal{B}(A,V^{\#})^{\mathrm{op}})$ and $H_{r}({}_\square \mathcal{B}(A,V^{\#}))^{\mathrm{op}}$ are isomorphic. This leads to the following isomorphism of Hopf algebras:
\begin{eqnarray*}
{}_\square \mathcal{H}(A,V^{\#})^{\mathrm{op}} = H_{r}\bigl({}_\square \mathcal{B}(A,V^{\#})\bigl)^{\mathrm{op}} \cong H_{r}\bigl({}_\square \mathcal{B}(A,V^{\#})^{\mathrm{op}} \bigl)  \stackrel{{\rm Theorem}\, \ref{dual1}}{\cong}  H_{r}\bigl({}_\square \mathcal{B}(A^*,V)\bigl) = {}_\square \mathcal{H}(A^*,V)
\end{eqnarray*}
and the proof is now finished.
\end{proof}

Analogously, we have

\begin{theorem}\label{dual2} Let $A$ be a finite dimensional $\Omega$-algebra and let $V \subseteq \End_F(A)$ be a unital subalgebra. Then we have a bialgebra isomorphism $$\mathcal{B}^\square(A^*,V)
\cong \mathcal{B}^\square(A,V^{\#})^{\mathrm{cop}}.$$	
\end{theorem}
\begin{proof}
Consider the category $\mathbf{Coactions}(A^*,V)$ where the objects are all coactions $\rho \colon  A^* \to A^* \otimes B$ for all bialgebras $B$ with $\cosupp \rho \subseteq V$ and the morphisms from $\rho_1 \colon A^* \to A^* \otimes B_1$
to $\rho_2 \colon A^* \to A^* \otimes B_2$  are bialgebra homomorphisms $\varphi \colon B_1 \to B_2$
making the diagram below commutative:

$$\xymatrix{ A^* \ar[r]^(0.4){\rho_1} 
\ar[rd]_{\rho_2}
 & A^* \otimes B_1 \ar[d]^{\id_{A^*} \otimes \varphi} \\
& A^* \otimes B_2} $$

Then the initial object of $\mathbf{Coactions}(A^*,V)$ is the coaction
of the bialgebra $\mathcal{B}^\square(A^*,V)$ on $A$.

Analogously, the coaction of the bialgebra $\mathcal{B}^\square(A,V^{\#})$ is the initial object in the category $\mathbf{Coactions}(A,V^{\#})$
where the objects are all coactions $\rho \colon  A \to A \otimes B$ for all bialgebras $B$ with $\cosupp \rho \subseteq V^{\#}$ and the morphisms from $\rho_1 \colon A \to A \otimes B_1$
to $\rho_2 \colon A \to A \otimes B_2$  are bialgebra homomorphisms $\varphi \colon B_1 \to B_2$
making the diagram below commutative:

$$\xymatrix{ A \ar[r]^(0.4){\rho_1} 
\ar[rd]_{\rho_2}
 & A \otimes B_1 \ar[d]^{\id_{A} \otimes \varphi} \\
& A \otimes B_2} $$

Let $\widehat{\ }$ be the functor $\mathbf{Coactions}(A^*,V) \to \mathbf{Coactions}(A,V^{\#})$
defined as follows: if $\rho \colon  A^* \to A^* \otimes B$ is an object of $\mathbf{Coactions}(A^*,V)$,
then $\widehat\rho \colon  A \to A \otimes B^\mathrm{cop}$ is defined by
$$(a^*\otimes \id_B)\widehat\rho(a):=\rho(a^*)(a)\text{ for all }a\in A,\ a^*\in A^*.$$

Note that $\widehat\rho^{\vee}(b^*\otimes(-)) = \left(\rho^{\vee}(b^*\otimes(-))\right)^*$
for all $b^*\in B^*$ where $\rho^{\vee}(b^*\otimes a^*):=b^*(a^*_{(1)})a^*_{(0)}$.
Hence $\cosupp \widehat\rho = (\cosupp\rho)^{\#}$. The $\Omega$-algebra $A$ becomes a left $B$-comodule and therefore a right $B^\mathrm{cop}$-module.

A straightforward check shows that $\widehat{\ }$ is an isomorphism of categories and therefore 
$\widehat{\ }$ maps the initial object to the initial object. Now the theorem follows.
\end{proof}

\begin{theorem}\label{dual4}
For any finite dimensional $\Omega$-algebra $A$ such that $\mathcal{B}^\square(A,V^{\#})$ is a Hopf algebra, we have an isomorphism of Hopf algebras $$\mathcal{H}^\square(A^*,V)
\cong \mathcal{H}^\square(A,V^{\#})^{\mathrm{cop}}.$$
\end{theorem}
\begin{proof}
First note that our assumption implies that the bialgebra $\mathcal{B}^\square(A,V^{\#})^{\mathrm{cop}}$ has a skew antipode. In light of \cite[Lemma 3.6]{CH} the Hopf algebra $H_{l}\bigl(\mathcal{B}^\square(A,V^{\#})^{\mathrm{cop}}\bigl)$ has a bijective antipode. 
Having in mind that the embedding functor $\mathbf{Hopf}_F \to \mathbf{Bialg}_F$ is right adjoint to $H_{l}$ it follows that for any Hopf algebra with bijective antipode $H$ we have the following natural bijections:
\begin{eqnarray*}
&&\Hom_{\,\mathbf{SHopf}_{F}}\bigl(H_{l}(\mathcal{B}^\square(A,V^{\#})^{\mathrm{cop}}),\, H\bigl) = \Hom_{\,\mathbf{Hopf}_{F}}\bigl(H_{l}(\mathcal{B}^\square(A,V^{\#})^{\mathrm{cop}}),\, H\bigl)\\ 
&& \cong \Hom_{\,\mathbf{Bialg}_{F}} \bigl(\mathcal{B}^\square(A,V^{\#})^{\mathrm{cop}},\, H\bigl) =  \Hom_{\,\mathbf{Bialg}_{F}} \bigl(\mathcal{B}^\square(A,V^{\#}),\, H^{\mathrm{cop}}\bigl)\\
&& \cong \Hom_{\,\mathbf{Hopf}_{F}}\bigl(H_{l}(\mathcal{B}^\square(A,V^{\#})),\, H^{\mathrm{cop}}\bigl) = \Hom_{\,\mathbf{Hopf}_{F}}\bigl(H_{l}(\mathcal{B}^\square(A,V^{\#}))^{\mathrm{cop}},\, H\bigl)\\ 
&& = \Hom_{\,\mathbf{SHopf}_{F}}\bigl(H_{l}(\mathcal{B}^\square(A,V^{\#}))^{\mathrm{cop}},\, H\bigl).
\end{eqnarray*}
Therefore, $H_{l}(\mathcal{B}^\square(A,V^{\#})^{\mathrm{cop}})$ and $H_{l}(\mathcal{B}^\square(A,V^{\#}))^{\mathrm{cop}}$ are isomorphic Hopf algebras. Furthermore, we have the following isomorphism of Hopf algebras:
\begin{eqnarray*}
\mathcal{H}^\square(A,V^{\#})^{\mathrm{cop}} = H_{l}\bigl(\mathcal{B}^\square(A,V^{\#})\bigl)^{\mathrm{cop}} \cong H_{l}\bigl(\mathcal{B}^\square(A,V^{\#})^{\mathrm{cop}}\bigl) \stackrel{{\rm Theorem}\, \ref{dual2}}{\cong} H_{l}(\mathcal{B}^\square(A^*,V)) = \mathcal{H}^\square(A^*,V)
\end{eqnarray*}
and the theorem now follows.

\end{proof}

\subsection{Examples of (co)algebras for which the universal coacting Hopf algebra does not exist}\label{non-exist}

As it was originally shown by Tambara and follows from
Corollary~\ref{CorollaryManinExistence} and Examples~\ref{ExampleAlgebra}--\ref{ExampleUnitalAlgebra}, the universal coacting Hopf algebra of a finite dimensional algebra always exists. Moreover, as noted in Examples~\ref{exUniv1}, 3) and~\ref{exUniv2}, 3), the universal coacting Hopf algebra also exists for any finite dimensional coalgebra.
However, this is not necessarily the case for arbitrary (co)algebras as the next examples show.

\begin{example}
Let $A$ be the algebra over an algebraically closed field $F$ of characteristic $0$, with the countable basis $1_A, v_1, v_2, v_3, \ldots$
such that $v_i v_j = 0$ for all $i,j \in\mathbb N$.
Then there exists neither $\mathcal{B}^{\square}(A)$ nor $\mathcal{H}^{\square}(A)$.
\end{example}
\begin{proof}
Let $C_n=\langle c_n \rangle_n$ be the cyclic group of order $n$, $n\in\mathbb N$. Consider the following $C_n$-action on $A$ by automorphisms: $$c_n v_i = \left\lbrace \begin{array}{lll}
 v_{i+1} & \text{ if } &  i<n, \\
 v_1 & \text{ if } &  i=n,\\
 v_i & \text{ if } &  i>n.
  \end{array} \right.$$
  Let $\zeta_n$ be the primitive $n$th root of unity.
It is easy to see that $$1_A,$$
   $$w_{nj}=\sum_{i=1}^n \zeta_n^{(i-1)(j-1)} v_i \text{ for }j=1,\ldots, n,$$
  $$v_{n+1}, v_{n+2},\ldots$$ form a basis of eigenvectors for this $C_n$-action.
  Moreover $c_n w_{nj} = \zeta_n^{1-j} w_{nj} = \chi_n^{1-j}(c_n) w_{nj}$.
  
  Let $C^*_n = \Hom(C_n, F^\times) = \langle \chi_n \rangle_n$ be the dual group of $C_n$,
  $\chi_n(c_n):= \zeta_n$. Consider the $C^*_n$-grading
  $A=\bigoplus_{j=0}^n A^{\left(\chi_n^j \right)}$
   where $w_{nj} \in A^{\left(\chi_n^{1-j} \right)}$
   and $1_A, v_{n+1}, v_{n+2},\ldots \in A^{\left(1\right)}$.
   
   The $C_n$-action and the $C^*_n$-grading defined above correspond to the  $FC^*_n$-comodule structure $\rho_n \colon A \to A \otimes FC^*_n$ defined by $\rho_n(1_A)=1_A\otimes 1$, $\rho_n(w_{nj}) := w_{nj} \otimes \chi_n^{1-j}$ for $j=1,\ldots, n$ and $\rho_n(v_i)=v_i \otimes 1$
   for $i > n$. Note that 
   $v_1 = \frac{1}{n}\sum_{j=1}^n w_{nj}$  
      and $$\rho_n(v_1)= \frac{1}{n} \sum_{j=1}^n w_{nj} \otimes \chi_n^{1-j}
      = \frac{1}{n} \sum_{j=1}^n \sum_{i=1}^n \zeta_n^{(i-1)(j-1)} v_i \otimes \chi_n^{1-j}
      = \sum_{i=1}^n v_i \otimes \chi_{ni}$$
      where $\chi_{ni} := \frac{1}{n} \sum_{j=1}^n \zeta_n^{(i-1)(j-1)} \chi_n^{1-j}$, $1\leqslant i \leqslant n$.
      Now suppose there exist $\mathcal{B}^{\square}(A)$ and denote by $\rho \colon A \to A \otimes
      \mathcal{B}^{\square}(A)$ the corresponding coaction.
      We have $\rho(v_1) = 1_A \otimes h_0 + \sum_{i=1}^m v_i \otimes h_i$
      for some $m\in\mathbb N$ and some $h_i \in \mathcal{B}^{\square}(A)$.
      Then, as for fixed $n$ the elements $\chi_{ni}$, where $i=1,\ldots,n$, are linearly independent by the Vandermonde argument, for $n > m$ there exist no bialgebra homomorphism $\varphi \colon \mathcal{B}^{\square}(A) \to
      FC^*_n$ making the diagram below commutative:
      $$\xymatrix{ A \ar[r]^(0.3){\rho} \ar[rd]_{\rho_n} &  A \otimes \mathcal{B}^{\square}(A) \ar@{-->}[d]^{\id_A \otimes \varphi} \\ 
              &  A \otimes FC^*_n}$$
              The same argument works for $\mathcal{H}^{\square}(A)$.      \end{proof}
      
     \begin{remark}
     The same proof works for $A=F[v_1, v_2, v_3, \ldots]$ which has no universal coacting Hopf and bialgebras either.
     \end{remark}
      
\begin{example}
Let $(C,\Delta,\varepsilon)$ be the coalgebra over an algebraically closed field $F$ of characteristic $0$, with the countable basis $v_0, v_1, v_2, v_3, \ldots$
such that $\varepsilon(v_0)=1$, $\Delta(v_0)=v_0 \otimes v_0$,
$\Delta(v_i)=v_0 \otimes v_i + v_i\otimes v_0$, $\varepsilon(v_i) = 0$  for all $i \in\mathbb N$.
Then there exists neither $\mathcal{B}^{\sqbullet}(C)$ nor $\mathcal{H}^{\sqbullet}(C)$.
\end{example}
\begin{proof} We consider the $C_n$-action 
$$c_n v_i = \left\lbrace \begin{array}{lll}
 v_{i+1} & \text{ if } &  i<n, \\
 v_1 & \text{ if } &  i=n,\\
 v_i & \text{ if } &  i>n \text{ or } i=0
  \end{array} \right.$$
and repeat verbatim the proof of the previous example.
\end{proof}

\end{document}